\documentclass[a4paper,11pt,leqno]{article}

\usepackage[T5,T1]{fontenc}
\usepackage[utf8]{inputenc}

\usepackage{amsmath}
\usepackage{amsfonts}
\usepackage{amssymb}
\usepackage{amsthm}
\usepackage{thmtools}
\usepackage{mathabx}
\usepackage{comment}
\usepackage{svg}
\usepackage{tikz-cd}
\tikzcdset{arrow style=math font} 
\usepackage{textcomp}

\usepackage{ifpdf}
\usepackage{ifthen}

\usepackage{enumerate}

\usepackage{aliascnt}

\usepackage{color}
\usepackage[hypertexnames=false]{hyperref} %Hyperlink reference
\hypersetup{          %Set up  for hyperlinks
  colorlinks=true,
  breaklinks,
  linkcolor=blue,
  filecolor=magenta, urlcolor=blue,
  citecolor=magenta, linktoc=all, }
\usepackage{cleveref}

\crefformat{enumi}{#2\textup{#1}#3}

\usepackage[all]{xy}

\usepackage{pifont}
\usepackage{enumitem}
\usepackage[nottoc,numbib]{tocbibind}

\usepackage{tikz}

\usepackage{footnote}

\usepackage{multirow}

\numberwithin{equation}{section} % Number equations within sections (i.e. 1.1,
\numberwithin{figure}{section} % Number figures within sections (i.e. 1.1, 1.2,
\numberwithin{table}{section} % Number tables within sections (i.e. 1.1, 1.2,

\newtheorem{thm}[equation]{Theorem}
\Crefname{thm}{Theorem}{thm}

\newtheorem*{ack}{Acknowledgments}

\newtheorem{thmA}{Theorem}

\newtheorem{conjA}[thmA]{Conjecture}

\newtheorem{lemma}[equation]{Lemma}
\Crefname{lemma}{Lemma}{Lemmas}

\newtheorem{prop}[equation]{Proposition}
\Crefname{prop}{Proposition}{thm}

\newtheorem{cor}[equation]{Corollary}
\Crefname{corollary}{Corollary}{corollary}

\newtheorem{conj}[equation]{Conjecture}
\Crefname{conjecture}{Conjecture}{conjecture}

\theoremstyle{definition}

\newtheorem{rem}[equation]{Remark}

\newtheorem{example}[equation]{Example}

\newtheorem{definition}[equation]{Definition}
\newtheorem{question}[equation]{Question}

\Crefname{notation}{Notation}{notation}

\newtheorem{block}[equation]{}

\crefname{subsection}{subsection}{subsections}
\Crefname{subsection}{Subsection}{Subsections}

\Crefname{figure}{Fig.}{Fig.}

\crefformat{figure}{fig.~#2#1#3}
\Crefformat{figure}{Fig.~#2#1#3}

\crefformat{equation}{eq.~#2#1#3}
\Crefformat{equation}{Eq.~#2#1#3}

\crefformat{enumi}{#2#1#3}
\crefrangeformat{enumi}{#3#1#4-#5#2#6}

\crefname{thm}{Theorem}{thm}
\crefname{thmA}{Theorem}{thm}
\crefname{lemma}{Lemma}{thm}
\crefname{cor}{Corollary}{thm}
\crefname{definition}{Definition}{thm}
\crefname{rem}{Remark}{thm}
\crefname{conj}{Conjecture}{thm}
\crefname{example}{Example}{thm}
\crefname{question}{Question}{thm}
\crefformat{figure}{Fig.~#2#1#3}

\newcommand{\C}{\mathbb{C}}
\newcommand{\CP}{\mathbb{CP}}

\newcommand{\Z}{\mathbb{Z}}
\newcommand{\R}{\mathbb{R}}
\newcommand{\Q}{\mathbb{Q}}

\newcommand{\bZ}{\overline{\Z}}
\newcommand{\bQ}{\overline{\Q}}
\newcommand{\bR}{\overline{\R}}

\newcommand{\Cinf}{\mathcal{C}^\infty}

\newcommand{\tX}{\tilde X}

\newcommand{\tY}{\tilde Y}
\newcommand{\tH}{\tilde H}

\newcommand{\TPS}{\mathop{\mathrm{J}}\nolimits}
\newcommand{\tTPS}{\mathop{\mathrm{AJ}}\nolimits}
\newcommand{\CF}{\mathop{\mathcal{F}}\nolimits}
\newcommand{\CFne}{\mathop{\mathcal{F}}_{\mathrm{ne}}\nolimits}
\newcommand{\NCF}{\mathop{\mathcal{N}}\nolimits}
\newcommand{\CNN}{\mathop{\mathrm{CN}_{\mathfrak{N}}}\nolimits}
\newcommand{\Tne}{\mathcal{T}_{\mathrm{ne}}}
\newcommand{\CN}{\mathop{\mathrm{CN}}\nolimits}
\newcommand{\capno}{\mathop{\mathrm{Cap}}\nolimits}
\newcommand{\maxax}{\mathop{\mathcal{M}}\nolimits}

\newcommand{\Loj}{\mathop{\mathcal{L}}\nolimits}
\newcommand{\new}{\mathop{\mathfrak{N}}\nolimits}
\newcommand{\Knew}{\mathop{K\mathfrak{N}}\nolimits}
\newcommand{\sed}{\mathop{\mathrm{sed}}\nolimits}
\newcommand{\lc}{\mathop{\mathrm{lc}}\nolimits}

\renewcommand{\O}{\mathcal{O}}

\newcommand{\m}{\mathfrak{m}}
\newcommand{\aaa}{\mathfrak{a}}
\newcommand{\T}{\mathcal{T}}

\newcommand{\rk}{\mathop{\mathrm{rk}}}
\newcommand{\im}{\mathop{\mathrm{im}}}

\newcommand{\mult}{\mathop{\mathrm{mult}}\nolimits}

\newcommand{\Vol}{\mathop{\rm Vol}\nolimits}
\newcommand{\convx}{\mathop{\rm conv}\nolimits}
\newcommand{\Gr}{\mathop{\rm Gr}\nolimits}

\newcommand{\supp}{\mathop{\rm supp}\nolimits}

\newcommand{\GL}{\mathop{\rm GL}\nolimits}

\newcommand{\ord}{\mathrm{ord}}

\newcommand{\newpol}[2]{
  \left\{
  \begin{array}{c}
  #1 \\ \hline \hline
  #2
  \end{array}
  \right\}
}

\newcommand{\mleq}{<\hspace{-.15cm}<}

\newcommand{\set}[2]{\left\{ #1 \,\middle\vert\, #2 \right\}}
\newcommand{\gen}[2]{\left\langle #1 \,\middle|\, #2 \right\rangle}

\renewcommand{\epsilon}{\varepsilon}

\newcommand{\fa}[2]{\forall #1 \,:\, #2}

\usepackage{geometry}

\newcounter{dummy}
\renewcommand{\thedummy}{\roman{dummy}}
\crefformat{dummy}{#2(#1)#3}

\newenvironment{blist}
{
	\begin{list}{(\thedummy)}
		{
			\setlength\labelsep{4pt}
			\setlength\itemindent{4pt}
			\setlength\leftmargin{0pt}
			\setlength\labelwidth{0pt}
			\usecounter{dummy}
		}
	}
	{
	\end{list}
}

\newcommand{\sectionbreak}{
  \begin{center}
    \ding{96} \ding{96} \ding{96}
  \end{center}
}

\author{
{\fontencoding{T1}\selectfont
Baldur Sigurðsson}\footnote{
The author was partially supported by the
Spanish grant of MCIN (PID2020-114750GB-C32/AEI/10.13039/501100011033)
and a
Mar\'ia Zambrano contract for international atraction of talent at
Universidad Complutense de Madrid.
}
\footnote{
Universidad Polit{\'e}cnica de Madrid,
Matem\'atica e Inform\'atica Aplicadas a las Ingenier\'ias Civil y Naval,
C. del Profesor Aranguren 3, 28040 Madrid.
\tt{baldursigurds@gmail.com}}
}

\title{
On the Jacobian polygon and \L ojasiewicz exponent
of isolated complex hypersurface singularities
}
\date{\today}

\begin{document}
\maketitle

\begin{abstract}

Given an isolated
hypersurface singularity $(X,0) \subset (\C^{n+1},0)$ defined by
a holomorphic function $f:(\C^{n+1},0) \to (\C,0)$,
we introduce an alternating version of Teissier's Jacobian Newton polygon,
and prove formulas
for both these invariants in terms of an embedded resolution of $(X,0)$.
The formula for
the alternating version has an advantage, in that for Newton nondegenerate
functions, it can be calculated in terms of volumes of faces of the Newton
diagram, whereas a similar formula for the original non-alternating version
includes mixed volumes.

The Milnor fiber can be given a handlebody decomposition, with handles
corresponding to intersection points with the polar curve in generic plane
sections of the singularity. This way we obtain a Morse--Smale complex.
Teissier
associates with each branch of the polar curve a vanishing rate, and we show
that this induces a filtration of the Morse--Smale complex. We apply this result
in order to calculate the {\L}ojasiewicz exponent in terms of the alternating
Jacobian polygon, but we expect it to be of further independent interest.
In the case
of a Newton nondegenerate hypersurface, our result produces a formula for the
{\L}ojasiewicz exponent in terms of Newton numbers of certain subdiagrams.

This statement is related to a conjecture by Brzostowski, Krasi\'nski and
Oleksik, for which we provide a counterexample. Our formula for the
\L ojasiewicz exponent is based on a global calculation over the Newton
diagram, rather than locally specifying a subset of the facets to consider,
as in this conjecture. We conjecture a similar statement, which is based on
our formula and inspired by the nonnegativity of local $h$-vectors.

\end{abstract}
\tableofcontents

\section{Introduction} \label{s:intro}

\begin{block}
In \cite{Tei}, Teissier associates a Newton polygon with a complex
isolated hypersurface singularity $(X,0) \subset (\C^{n+1},0)$,
defined as the zero set
of a holomorphic function $f \in \C\{x_0,x_1,\ldots,x_n\}$.
This is the Newton polygon of \emph{the Cerf diagram},
the image of the polar curve under the map $(f,g):(\C^{n+1},0) \to (\C^2,0)$,
where $g:\C^{n+1} \to \C$ is a generic linear function.
We will refer to this polygon as the Jacobian polygon, and
denote it $\TPS(f,0)$.
The set of Newton polygons $\new$ is a commutative and
cancellative semigroup, generated by diagrams of the form
\begin{minipage}{.5\textwidth}
\begin{equation} \label{eq:intro_newpol}
  \newpol{m}{n} = \Gamma_+(x^m + y^n) \subset \R_{\geq 0}^2.
\end{equation}
\end{minipage}
\begin{minipage}{.5\textwidth}
\vspace{0.2cm}
\begin{center}
\begin{picture}(0,0)%
\includegraphics{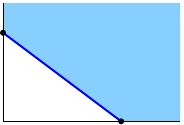}%
\end{picture}%
\setlength{\unitlength}{4144sp}%
\begin{picture}(1385,940)(428,-539)
\put(1351,-466){\makebox(0,0)[lb]{\smash{\fontsize{12}{14.4}\normalfont {\color[rgb]{0,0,0}$x^m$}%
}}}
\put(496,209){\makebox(0,0)[lb]{\smash{\fontsize{12}{14.4}\normalfont {\color[rgb]{0,0,0}$y^n$}%
}}}
\end{picture}%

\label{fig:mnpolyg}
\end{center}
\vspace{0.0cm}
\end{minipage}
Of particular interest is the highest number $m/n$ appearing in $\TPS(f,0)$,
which we will call its \emph{degree}.
By \cite{Tei,lejjaltei}, this number coincides with the
\emph{\L ojasiewicz exponent}.
This work is motivated by the search of an identification
of the \L ojasiewicz exponent of a Newton nondegenerate function
in terms of its Newton diagram, which has been conducted in
\cite{bko_conj,bko_surface,
fukui_newton_diag,lichtin_estimation,lenarcik_exponent}.
This problem---and in particular, the conjecture of
Brzostowski, Krasi\' nski and Oleksik, counter to which we
give \cref{example:counter}---is discussed further below.
This this problem can be seen as a version of one of Arnold's
problems \cite[1975-1]{arnolds_problems}.

We introduce \emph{the alternating Jacobian polygon} $\tTPS(f,0)$
(\cref{def:TN}\cref{it:TN_AJ}) which lives in the
group $K\new$ of formal differences of Newton polygons,
and show that this invariant
can be calculated from an embedded resolution.
The alternating Jacobian polygon is the alternating sum of Jacobian
polygons
$\TPS\left(f^{(d+1)},0\right)$,
where $f^{(d+1)}$ is the restriction
of $f$ to a generic linear subspace of dimension $d+1$.
As a result, one recovers
$\TPS\left(f,0\right)
 = \tTPS\left(f^{(n+1)},0\right) + \tTPS\left(f^{(n)},0\right)$.
At any smooth point of the total transform of $X$ in an embedded
resolution of $(X,0)$, we can define a number $m$, the order of vanishing
of the pullback of $f$, and similarly $n$, the order of vanishing of
a generic linear function.
The formula \cref{eq:ttps} for $\tTPS(f,0)$ can be seen as an integral
of the additive Newton polygon \cref{eq:intro_newpol} with respect to
the Euler characteristic,
in a similar way that A'Campo's formula \cite{ACampo}
calculates the monodromy zeta
function as an integral of the multiplicative polynomial $(t^m - 1)^{-1}$.
The strong relationship between the Jacobian and alternating Jacobian
polygons means that we recover a similar
formula for Teissier's Jacobian polygon as well.
\end{block}

\begin{thmA}[\ref{thm:tps}] \label{thmA:tps}
Let $(Y,D) \to (\C^{n+1},0)$ be an embedded resolution of $(X,0)$.
For each irreducible component $D_i \subset D$ of the exceptional
divisor, denote by $m_i$ and $n_i$ the order of vanishing along $D_i$ of the
pullback of $f$ and a generic linear function, respectively,
and denote by $H \subset \C^{n+1}$ a generic linear hyperplane.
The Jacobian and alternating Jacobian polygons can be computed from
an embedded resolution of $(X,0)$ as
\[
  \TPS(f,0)
  =
  (-1)^n \sum_{i \in I}
                \chi(D_i^\circ \setminus (\tX\cup\tH))
                \newpol{m_i}{n_i}
\]
and
\[
  \tTPS(f,0)
  =
  (-1)^n \sum_{i \in I}
                \chi(D_i^\circ \setminus \tX)
                \newpol{m_i}{n_i},
\]
where $\tX$ and $\tH$ denote the strict transforms of $X$ and $H$ in $Y$.
\end{thmA}
\begin{rem}
This statement, in the case of plane curves, $n=1$, can be seen 
to follow from a more general result of Michel \cite{Mich_Jac}, where
she considers arbitrary finite holomorphic maps $(X,p) \to (\C^2,0)$
in place of our map $(f,g):(\C^2,0)\to (\C^2,0)$, where $(X,p)$ is
a normal surface singularity.
We expect a similar generalization to hold for finite maps
$(X,p) \to \C^2$ with $(X,p)$ isolated of any dimension.
\end{rem}

\begin{block}
As an application of this theorem, we find a formula for the
alternating Jacobian polygon in terms of volumes of faces
of the Newton diagram,
in the case of a Newton nondegenerate hypersurface,
\cref{thm:tps_nondeg}\cref{it:tps_nondeg_n}, reflecting
Varchenko's formula \cite{Var_zeta} for the monodromy zeta function.
We also calculate the
alternating Jacobian polygons associated with a generic plane section
of any codimension in terms of mixed volumes.
As a result, we find a formula for the Jacobian polygon associated
with any plane section of $(X,0)$, reflecting Oka's work on
principal zeta functions \cite{oka_principal}.
Here, a facet of a Newton diagram $\Gamma \in \R^{n+1}$ is a face
of dimension $n$, and a coordinate facet is a facet of a diagram obtained
by intersecting $\Gamma$ with a coordinate subspace.
\end{block}

\begin{thmA}[\ref{thm:tps_nondeg}, \ref{cor:tps_nondeg}]
\begin{blist}
\item
Let $f:(\C^{(n+1)},0) \to (\C,0)$ be a Newton nondegenerate
holomorphic function with an isolated singularity. Then, the
alternating Jacobian polygon $\tTPS(f,0)$ can be calculated
in terms of volumes of coordinate facets of
the Newton diagram $\Gamma(f)$.
\item
The alternating Jacobian polygons $\tTPS(f^{(d+1)},0)$ and the
Jacobian polygons $\TPS(f^{(d+1)},0)$ of the restriction of $f$
to generic linear subspaces of dimension $d+1$ can be calculated
in terms of mixed volumes of coordinate facets
of the Newton diagrams $\Gamma(f)$
and $\Gamma(g)$, where $g:\C^{n+1} \to \C$
is a generic linear function.
\end{blist}
\end{thmA}

\begin{block}
The \emph{{\L}ojasiewicz exponent} $\Loj(f,0)$ of $f$ at $0$ is,
by definition, the
infimum of all $\theta > 0$ such that there exists a $C > 0$ such that
the inequality
\[
  \|x\|^\theta \leq C \|\nabla f\|
\]
holds near $0\in \C^{n+1}$.
Teissier proved in \cite{Tei} that
this invariant can be calculated directly as the degree of the
Jacobian polygon minus one (see \cref{def:polygons}).
See \cite{lejjaltei} for more on this topic.
We prove that the degree of the alternating Jacobian polygon
coincides with that of the Jacobian polygon.
The proof of this statement requires an analysis of the
topology of the Milnor fibration, along with that of generic plane sections,
which we expect to be of independent interest.
We study the Morse--Smale complex of a certain Morse function
on the Milnor fiber. It is known that
the Milnor fiber can be constructed as a handlebody, whose
number of handles of index $d$ is $\mu^{(d+1)} + \mu^{(d)}$.
Our Morse function has the same number of Morse points.
In \cite{Tei}, Teissier associates a rate of vanishing to each
branch of the polar curve. Following this construction, we 
define a filtration on the Morse--Smale complex as an Abelian group,
and we show that this is a filtration of the complex.
\end{block}

\begin{thmA}[\ref{lem:trajs} and \ref{cor:filtration}] \label{thmA:trajs}
The vanishing rates associated to each branch of the polar curve
induce an increasing filtration on the Morse--Smale complex
on the Milnor fiber, generated by intersection points with
polar curves in generic plane sections.
\end{thmA}

\begin{block}
This result allows us to conclude that the Jacobian and alternating Jacobian
polygons have the same degree. In particular,
we can calculate the \L ojasiewicz exponent
purely in terms of the alternating Jacobian polygon.
In the Newton nondegenerate case, it can be computed in terms of
volumes of coordinate facets. In fact, if $F$ is a coordinate facet,
then it has a unique primitive integral normal vector
\[
  v = (v_0, v_1, \ldots, v_n) \in (\Z_{>0} \cup \{+\infty\})^{n+1}.
\]
Define the \emph{maximal axial number} associated with $v$ and $f$
as the weight of $f$ with respect to $v$, divided by that of a generic
linear function.
The maximal axial number of a coordinate facet is that of its normal vector.
For facets, this invariant of a Newton diagram has been studied in
\cite{bko_conj}.
By defining $s_\alpha(\Gamma(f)) \subset \Gamma(f)$ as the union
of faces supported by weight vectors having maximal axial number
$\leq \alpha$, for $\alpha \in \Q$,
we have an identification of the \L ojasiewicz exponent
in terms of Newton numbers of these subdiagrams.
\end{block}

\begin{thmA}[\ref{thm:bko_statement}] \label{thmA:bko}
With the possible exception of Morse points in an even number of variables,
the \L ojasiewicz exponent of an isolated Newton nondegenerate
singularity is the maximal axial number of some coordinate facet
of its Newton diagram. It can be recovered as
\[
  \Loj(f,0)
  =
  \min\set{ \alpha - 1 \in \Q }
          { \nu(s_\alpha(\Gamma(f))_-) = \nu(\Gamma_-) }.
\]
where $\nu$ is the Newton number.
\end{thmA}

\begin{rem}
Let $F$ be the Milnor fiber of an isolated singularity
$(X,0) \subset (\C^{n+1})$ defined by $f \in \C\{x_0,\ldots,x_n\}$.
Then we have
$(-1)^{n+1}\chi(F) \geq 0$, with equality if and only if $(X,0)$
is a Morse point and $n$ is odd.
In this text, we say that $f$ has a Morse point at $0$, or that $(X,0)$
is a Morse point, if one of the following equivalent conditions is satisfied:
\begin{itemize}
\item
$\mu(X,0) = 1$,
\item
$f$ has type $A_1$,
\item
after a change of coordinates, we have
$f(x_0,x_1,\ldots,x_n) =  x_0^2+x_1^2+\cdots+x_n^2$,
\item
the Hessian matrix of $f$ at $0$ is invertible.
\end{itemize}
\end{rem}

\begin{block}
In \cref{s:ppf} we recall the conjecture of
Brzostowski, Krasiński and Oleksik \cite{bko_conj}
on the \L ojasiewicz exponent of a Newton nondegenerate singularity $f$
which greatly motivated this manuscript.
As pointed out by the same authors in
\cite{bko_surface},
Arnold suggests in \cite[1975-1]{arnolds_problems} that
\emph{``Every interesting discrete invariant of a generic singularity
with Newton
polyhedron $\Gamma$ is an interesting function of the polyhedron.''}
In the current case of study, the discrete invariant is the \L ojasiewicz
exponent. 
The conjecture says that if $\Gamma(f)$ has a \emph{nonexceptional}
facet (see \cref{def:bko_exc}), then the \L ojasiewicz exponent of $f$ is
the largest maximal axial number of a nonexceptional facet minus one.
Stated differently, for any facet $F\subset\Gamma(f)$
of the Newton diagram, the
conjecture gives a simple condition, which only depends on $F$, and not
$\Gamma(f)$,
determining whether $F$ should be discarded or not.
Unless all facets have been discarded this way, the conjecture says that
the \L ojasiewicz exponent is the largest of the maximal axial numbers
of remaining facets, minus one.
They proved the conjecture in the case $n=2$.
We give a counterexample to this conjecture with $n=3$.
This example contains a $B_2$-facet, connecting our work to recent
advances on the monodromy conjecture
\cite{ELT_monodromy,LPS_motivic,Selyanin_B} for
Newton nondegenerate singularities.

In light of \cref{thmA:bko},
we ask what condition characterizes those \emph{coordinate facets}
of $\Gamma(f)$ (not just \emph{facets})
to include or exclude in determining the \L ojasiewicz exponent.
Motivated by the theory of local $h$-polynomials
\cite{stanley_subdivisions,selyanin_mono},
we define a subset $\CF_\mathrm{ne}^\T$ of the set $\CF(\Gamma(f))$
of coordinate facets of $\Gamma(f)$ which depends on a triangulation
$\T$ of $\Gamma(f)$.
\end{block}

\begin{conjA} \label{conjA:conj}
Let $f \in \C\{z_0,z_1,\ldots,z_n\}$
be Newton nondegenerate, defining an isolated singularity, and
let $\T$ be a triangulation of its Newton diagram $\Gamma(f)$.
Then the formula
\[
  \Loj(f,0) = \max_{F \in \CF_{\mathrm{ne}}^\T} \maxax(F) - 1.
\]
holds, with the possible exception of a Morse point in an even
number of variables.
\end{conjA}

\sectionbreak

\begin{block}
In \cref{s:polygon}, we review the notation for Newton
polygons introduced by Teissier, along with some basic definitions.

In \cref{s:res}, we prove a technical result on resolutions which
we apply in proofs in the following sections, based
on the notion of toroidal embeddings \cite{KKMSD}.

In \cref{s:jacob} we review Teisser's definition of the Jacobian
polygon, and fix related notation for the polar curve.
We also define the alternating Jacobian polygon, and prove
a formula for both polygons in terms of an embedded resolution and
a generic linear function.

In \cref{s:vanishing} we construct a Morse function on the Milnor fiber,
whose critical points coincide with intersection points with polar curves.
The main technical result of this section, \cref{lem:trajs}, shows that
we can
endow the associated Morse--Smale complex with an increasing
filtration, using Teissier's \emph{vanishing rate} \cite{Tei},
which is a rational number associated with each polar branch.

In \cref{s:nonneg} we obtain a string of corollaries of the above result.
Of particular interest is \cref{cor:lc_degs}\cref{it:both_degs}, which
says that the Jacobian and alternating Jacobian have the same degree.
As a result,
the alternating Jacobian polygon determines the \L ojasiewicz exponent.

In \cref{s:nondeg}, we consider the case
of a Newton nondegenerate singularity with Newton diagram $\Gamma(f)$.
Using the results in
\cref{s:jacob} we give formulas for the Jacobian and
alternating Jacobian polygons in terms of the Newton diagram.

In \cref{s:loj}, we obtain a formula
for the \L ojasiewicz exponent for a Newton nondegenerate
singularity in terms of Newton numbers of
subdiagrams (the Newton number was introduced by Kouchnirenko
in \cite{kouchnirenko}, see \cref{def:newton_number}).
\end{block}

\begin{ack}
I would like to thank
{\fontencoding{T5}\selectfont
Nguyễn Hồng Đức,
Nguyễn Tất Thắng,
}
Pablo Portilla Cuadrado,
Bernard Teissier,
Mark Spivakovsky,
and
Nero Budur
for useful discussions which contributed to this work, as well as
the reviewer for their invaluable remarks.
\end{ack}

\section{The group of Newton polygons} \label{s:polygon}

\begin{block}
In this section we fix notation for Newton polygons, seen
as additive invariants, as in \cite{Tei,teissier_jacobian}.
\end{block}

\begin{block}
We will denote by $\new$ the set of Newton polygons in $\R^2$. Thus,
an element of $\new$ is any integral polygon contained in
$\R^2_{\geq 0}$, whose recession cone is $\R^2_{\geq 0}$.
In particular, if $m,n$ are any two positive numbers, then we denote by
\begin{equation} \label{eq:newpolyg}
  \newpol{m}{n}
\end{equation}
the Newton polygon of the polynomial $x^m + y^n$.
Extending either $m$ or $n$ to infinity, and imagining that
$x^\infty = y^\infty = 0$, we define the polygons
\begin{equation} \label{eq:newpolyg_infty}
  \newpol{\infty}{n}, \qquad
  \newpol{m}{\infty}
\end{equation}
associated with the polynomials $y^n$ and $x^m$.
The set $\new$ has an operation which we denote by $+$,
the Minkowski sum, given by
\[
  \Gamma + \Delta = \set{g+d\in\R^2}{g\in \Gamma,\;d\in\Delta}.
\]
This way, we have a cancellative and commutative semigroup
$(\new,+)$,
generated by the
symbols \cref{eq:newpolyg,eq:newpolyg_infty}, modulo the relations
\[
  \newpol{cm}{cn} = c\newpol{m}{n},\qquad c\in \Z_{>0}.
\]
In particular, as a commutative semigroup, $\new$ is freely generated
by the symbols \cref{eq:newpolyg} for $m,n\in\Z_{>0}$ and
$\gcd(m,n) = 1$, as well as the two elements \cref{eq:newpolyg_infty}
with $m = 1$ and $n = 1$. Note that we simply set $c\infty = \infty$
for any positive integer $c$.

We denote by $\Knew$ the Grothendieck group of $\new$. This group then
has the same description as $\new$ in terms of generators and
relations, only as an Abelian group, rather than a commutative semigroup.
In particular, any element $N \in \Knew$ has a unique presentation as
\begin{equation} \label{eq:A_expansion}
  N = \sum_{\alpha \in \bQ_{\geq 0}} a_\alpha \{\alpha\},
\end{equation}
where the family of integers
\[
  (a_\alpha)_{\alpha \in \bQ_{\geq 0}}, \qquad
  \bQ_{\geq 0} = \Q_{\geq 0} \cup \{ \infty \}
\]
has finite support, and we set
\[
  \{\alpha\} = \newpol{m}{n}, \qquad
  \mathrm{if} \qquad
  \alpha = \frac{m}{n} \in \Q_{>0} \qquad
  \mathrm{with} \qquad
  \gcd(m,n) = 1
\]
and
\[
  \{0\} = \newpol{1}{\infty},\qquad
  \{\infty\} = \newpol{\infty}{1}.
\]
\end{block}

\begin{definition} \label{def:polygons}
\begin{itemize}
\item
The \emph{support} of an element $N \in \Knew$ expanded as
in \cref{eq:A_expansion} is
\[
  \supp(N)
  = \set{\alpha \in \bQ_{\geq0}}{a_\alpha \neq 0}
  \subset \bQ_{\geq 0}.
\]
\item
The \emph{degree} of an element $N \neq 0$ is
\[
  \deg(N)
  =
\begin{cases}
  \max \supp(N) & N \neq 0,\\
  -\infty       & N = 0.
\end{cases}
\]
\item
If $N \neq 0$, then the \emph{leading coefficient} of $N$ is
\[
  \lc(N) = a_\alpha.
\]
where $\alpha = \deg(N)$.
\item
If $\alpha \in \bQ_{\geq 0}$,
then the \emph{truncation of $N$ at $\alpha$} is
\[
  N_{\geq \alpha}
  = \sum_{\beta \geq \alpha} a_\beta \{\beta\}
  \in \Knew.
\]
\item
The \emph{height} $h(N) \in \bQ_{\geq 0}$ and \emph{length}
$\ell(N) \in \bQ_{\geq 0}$ of $N$ are defined by extending linearly
the functions
\[
  h\left(\newpol{m}{n}\right) = n,\qquad
  \ell\left(\newpol{m}{n}\right) = m.
\]
Note that this way, we have $h(N) = \pm \infty$
if and only if $a_0 \neq 0$,
and $\ell(N) = \pm \infty$ if and only if $a_\infty \neq 0$.
\end{itemize}
\end{definition}

\begin{block} \label{block:vertices}
In the above language, the vertices of a polygon $N \in \new$ are
precisely the points
\[
  \left(
    \ell(N_{<\alpha}), h(N_{\geq \alpha})
  \right), \qquad
  \alpha \in \Q_{>0}
\]
with adjacent points joined by an edge.
If $N \in K\new$, then these points and edges
 are the \emph{virtual vertices} and
\emph{virtual edges} of $N$.
\end{block}

\section{Embedded resolutions, coordinates and polar} \label{s:res}

\begin{block}
In this section we assume that $f \in \O_{\C^{n+1},0}$ is an
analytic function germ with an isolated critical point at the origin.
We describe some notation and conditions on an embedded resolution
of $f$.
\Cref{lem:Asep} is a technical result
which is used in the following sections.
\end{block}

\begin{block} \label{block:x_H}
Let $x_0, x_1, \ldots, x_n$ be linear coordinates in $\C^{n+1}$, i.e.
a basis of the dual space.
For $d = 0,1,\ldots,n$, define a linear subspace
\[
  H^{(d+1)} = \set{x\in \C^{n+1}}{x_{d+1} = \ldots = x_n = 0}.
\]
and set
\[
  f^{(d+1)} = f|_{H^{(d+1)}},\qquad
  g^{(d+1)} = x_d|_{H^{(d+1)}}.
\]
This way we have hypersurfaces $(X^{(d+1)},0) \subset (H^{(d+1)},0)$
defined by $f^{(d+1)}$, along with linear functions $g^{(d+1)}$.

If $\pi:Y \to \C^{n+1}$ is an embedded resolution of $f$, then
we denote by $Y^{(d+1)}$ the strict transform of $H^{(d+1)}$
in $Y$, and by
\[
  \pi^{(d+1)}:Y^{(d+1)} \to H^{(d+1)}
\]
the modification induced by $\pi$. Denote by $D^{(d+1)} \subset Y^{(d+1)}$
its exceptional divisor. In particular, we have
$\C^{n+1} = H^{(n+1)}$
and
$Y = Y^{(n+1)}$
and
$\pi = \pi^{(n+1)}$.
Denote also by $\tX^{(d+1)}$ the strict transform of $X^{(d+1)}$
in $Y^{(d+1)}$. Note that $\tX^{(d+1)}$ is not necessarily smooth, under
the conditions considered so far.

For $d=0,1,\ldots,n$, denote by $P^{(d+1)}$ the polar curve of $f^{(d+1)}$
with respect to the hyperplane $H^{(d)} \subset H^{(d+1)}$, that is
\[
  P^{(d+1)}
  = \set{x\in H^{(d+1)}}
        {\partial_{x_0}f = \ldots = \partial_{x_{d-1}}f = 0}.
\]
When $f^{(d+1)}$ has an isolated singularity, this analytic set is a curve.
Denote its branches by $P^{(d+1)}_1,\ldots,P^{(d+1)}_{l^{(d+1)}}$.
We also set $P = P^{(n+1)}$.

By \cite{Tei}, there exist numbers
\[
  \mu^{(d+1)}(f,0) \in \Z_{>0}, \qquad d=0,1,\ldots,n
\]
such that for a generic choice of coordinates,
$P^{(d+1)}$ is a reduced curve, and
\[
  \mu(f|_{H^{(d+1)}},0) = \mu^{(d+1)}(f,0).
\]
Let us say that such coordinates are
\emph{generic with respect to $f$}.
\end{block}

\begin{definition}
If $D_i\subset Y$ is a component of the exceptional divisor of a modification
$\pi:Y\to\C^{n+1}$,
and $h \in \O_{\C^{n+1},0}$, then we define $\ord_{D_i}(h)$ as the 
order of vanishing of $\pi^*h$ along $D_i$. If
$\aaa \subset \O_{\C^{n+1},0}$ is an ideal, then we set
\[
  \ord_{D_i}(\aaa)
  =
  \min\set{\ord_{D_i}(h)}{h \in \aaa}.
\]
\end{definition}

\begin{definition} \label{def:generic_Y}
Let $\pi:(Y,D) \to (\C^{n+1},0)$ be a point modification of $\C^{(n+1)}$
at $0$, whose exceptional divisor consists of components $D_i$, $i\in I$.
Let $x_0,\ldots,x_n$ and $H^{(d+1)}$ be as in \cref{block:x_H}.
\begin{enumerate}

\item
A function $h \in \m_{\C^{n+1},0}$ is
\emph{generic with respect to $\pi$} if
\[
  \fa{i\in I}{ \ord_{D_i}(h) = \ord_{D_i}(\m_{\C^{n+1},0}) }.
\]

\item
The coordinates $x_0,x_1,\ldots,x_n$ in $\C^{n+1}$ are generic
with respect to $\pi$ if for all $d = 0,1,\ldots,n$, the function
$g^{(d+1)} \in \m_{H^{(d+1)},0}$
is generic with respect to $\pi^{(d+1)}$, and for
$d>0$ the divisor
\[
  D^{(d+1)} \cup \tilde X^{(d+1)} \cup Y^{(d)} \subset Y^{(d+1)}
\]
is a normal crossing divisor.

\end{enumerate}
\end{definition}

\begin{lemma}
If $\pi$ factors through the blow-up of $\C^{n+1}$ at $0$, then
the set of coordinates generic with respect to $\pi$ contains a dense
Zariski open subset of $\GL(\C^{n+1})$.
\end{lemma}

\begin{proof}
This follows by applying Bertini's theorem to the linear system of
hyperplanes.
\end{proof}

\begin{block} \label{block:emb_res}
We will assume from now on that
$Y\to \C^{n+1}$ is an embedded resolution of $f$, and that
$x_0,\ldots,x_n$ are coordinates
in $\C^{n+1}$ which are generic with respect to
$f$ and $\pi$.
Decompose the exceptional divisor into irreducible
components as $D = \cup_{i\in I} D_i$, and set
\[
  D_J = \bigcap_{j\in J} D_j,\qquad
  D^\circ_J = D_J \setminus \bigcup_{j \notin J} D_j.
\]
for any $J \subset I$. For any $d=0,1,\ldots,n$, denote by
\[
  D^{(d+1)} = \bigcup_{i \in I^{(d+1)}} D_i
\]
the exceptional divisor of $\pi^{(d+1)}$, decomposed into its
irreducible components. We will assume that the sets $I^{(d+1)}$, for
different $d$, are chosen disjoint, so that we can freely refer
to $D_i \subset Y^{(d+1)}$, with $d$ depending on $i$.
We then define $D^{(d+1)}_J$ and $D^{(d+1)\circ}_J$ for $J \subset I^{(d+1)}$
similarly.
Since $D^{(d+1)} \cup Y^{(d)} \subset Y^{(d+1)}$
is a normal crossing divisor, we have well defined maps
\[
  \omega^{(d+1)} : I^{(d)} \to I^{(d+1)},\qquad
\]
defined by the condition that
\[
  D^\circ_i \subset D_{\omega(i)}^\circ.
\]
If $h \in \O_{\C^{n+1},0}$, then we define
\[
  \ord_{D_i}(h)
\]
the order of vanishing of $\pi^*(h)$ along $D_i$.
For any $d=0,1,\ldots,n$ and
$i \in I^{(d+1)}$, denote by $m_i, n_i$ the orders of vanishing of
$\pi^*f, \pi^*g$ along $D_i$, respectively, and set $\alpha_i = m_i / n_i$.
As divisors, we then have
\begin{equation} \label{eq:mn}
  (\pi^{(d+1)*}f^{(d+1)}) = \tX^{(d+1)} + \sum_{i\in I^{(d+1)}} m_i D_i,\qquad
  (\pi^{(d+1)*}g^{(d+1)}) = \tY^{(d)} + \sum_{i\in I^{(d+1)}} n_i D_i.
\end{equation}
Note that, since $D^{(d+1)} \cup \tX^{(d+1)} \cup Y^{(d)}$ is a normal
crossing divisor in $Y^{(d+1)}$, we have, for any $i \in I^{(d)}$
\[
  m_i = m_{\omega^{(d+1)}(i))},\qquad
  n_i = n_{\omega^{(d+1)}(i))}.
\]
If $J \subset I^{(d+1)}$, then we also set
\[
  m_J = \gcd\left(\set{m_j}{j\in J}\right), \qquad
  n_J = \gcd\left(\set{n_j}{j\in J}\right).
\]
\end{block}

\begin{definition}
Let $\alpha \in \Q_{>0}$ be given as a reduced fraction
$\alpha = m/n$ with $m,n > 0$, and $d \in \{0,1,\ldots,n\}$.
We define the meromorphic germ $\phi^{(d+1)}_\alpha$
at $(H^{(d+1)},0)$ as the fraction
\[
  \phi^{(d+1)}_\alpha
  =
  \frac{\left(g^{(d+1)}\right)^m}{\left(f^{(d+1)}\right)^n}.
\]
\end{definition}

\begin{definition}
Let $A \subset \Q_{>0}$ be a finite set of positive rational numbers.
The resolution $\pi$ is \emph{$A$-separating} if for every
$d=0,1,\ldots,n$, and every $\alpha \in A$,
the indeterminacy locus of the meromorphic function 
\[
  \left(\pi^{(d+1)}\right)^*
  \phi^{(d+1)}_\alpha
\]
is contained in $Y^{(d)} \cap \tilde X^{(d+1)}$.
\end{definition}

\begin{block} \label{block:Asep}
We end this section by recalling a construction from
\cite[Chapter II]{KKMSD}, which allows us to assume $A$-separatedness
for any finite set $A \subset \Q_{>0}$, and some properties.
Let $\pi:Y \to \C^{n+1}$ be a given resolution of $f$,
which factors through the origin, and assume that the coordinates
$x_0,x_1,\ldots,x_n$ in $\C^{n+1}$ are generic with respect to $\pi_0$ and
$f$.
Use the notation introduced so far in this section for $\pi$.
Denote by $\triangle$ and $\hat\triangle$ the
\emph{conical polyhedral complexes}
associated with the toroidal embeddings
\[
  Y \setminus (D \cup \tX) \subset Y, \qquad
  Y \setminus (D \cup \tX \cup \tilde Y^{(n)}) \subset Y,
\]
respectively,
as in \cite[Definition 5, II\textsection 1, p. 69]{KKMSD}.
What this means is the following:
Choose distinct elements $i_X, i_H$ not in $I^{(d+1)}$ for any $d$,
and set
$D_{i_X} = \tX \subset Y$
and $D_{i_H} = Y^{(n)} \subset Y_0$.
For any $J \subset I \cup \{i_X, i_H\}$ such that $D_J \neq \emptyset$,
we have objects
\[
  N_J = \Z^J, \qquad
  N_{J,\R} = \R^J, \qquad
  \sigma_J = \R^J_{\geq 0},
\]
along with natural inclusions between them whenever $K \subset J$.
Then $\hat\triangle$ is the conical polyhedral complex consisting of
all the cones $\sigma_J$ for such $J$, whereas
$\triangle$ is the subcomplex consisting of cones $\sigma_J$
with $i_H \notin J$.
We also define linear functions
\[
  u_{J,f}, u_{J,g}: N_{J,\R} \to \R
\]
by the requirements
\[
  u_{J,f}(i) = m_i,\qquad
  u_{J,g}(i) = n_i,\qquad
  i \in J.
\]
Note that these are compatible with inclusions $K \subset J$.

Given any subdivision $\triangle_A$ of $\triangle$ into cones,
Mumfords describes a map $Y_A \to Y$
which, in local coordinates compatible
with the divisor $D \cup \tX \cup \tilde Y^{(n)}$, is a toric map
\cite{KKMSD}.
Since this map is an isomorphism outside
$\cup_{i\neq j} (D_i \cap D_j)$, the composed map
$\pi:Y_A \to Y_0 \to \C^{n+1}$ remains an isomorphism outside $0 \in \C^{n+1}$.
Furthermore, if $\pi_A$ is constructed in this way, and $g = x_n$
is generic with respect to $\pi$, then, since $Y^{(n)}$ is transverse
to all the strata $D_J$, with $i_H \notin J$, no exceptional component
of $\pi_A$ in $Y_A$ maps into $Y^{(n)}$ via the map $Y_A \to Y$,
and we find that $g = x_n$ is generic with
respect to $\pi_A$ as well. Finally, since the restriction of
$\pi_A$ to the strict transform $Y_A^{(n)}$ of $H^{(n)}$ in $Y_A$ is in fact
obtained by the same construction, we find a finite iteration that
$x_0,\ldots,x_n$ are generic coordinates with respect to $\pi_A$.

Note that irreducible components $D_i$ of the exceptional divisor of $\pi$
correspond to rays, or one dimensional
cones, in $\triangle$, except for the one generated by $i_X$.
Similarly, irreducible exceptional components $D_{i,A}$
of $\pi_A$, indexed by a set $I_A$, correspond to rays in $\triangle$.
Thus, we have an inclusion $I \hookrightarrow I_A$, such that
if $i \in I$ then the map $Y_A \to Y$ induces a birational morphism
$D_{i,A} \to D_i$.
If $i \in I_A$ does not generate a ray in $\triangle$, then
there is a smallest $J \subset I$ such that
$\sigma_{\{i\}} \subset \sigma_J$, and in this case, the image of
$D_{i,A}$ in $Y$
is $D_J$, which has dimension $n + 1 - |J| < n$.
In either case, the induced map
$D_{i,A}^\circ \setminus \tX_A \to D_J^\circ \setminus \tX$,
where $\tX_A$ is the strict transform of $\tX$ in $Y_A$,
is a locally trivial fiber bundle with fiber $\left(\C^*\right)^{n+1-|J|}$,
which gives
\[
  \chi\left(D_{i,A}^\circ \setminus \tX_A\right)
  =
\begin{cases}
  \chi\left(D_i^\circ \setminus \tX \right) & |J| = 1, \\
  0 & |J| > 1.
\end{cases}
\]

Given a regular subdivision of $\hat \triangle_0$, we obtain a similar map
$\hat Y_A \to Y$ which induces a map
$\hat \pi_A:\hat Y_A \to \C^{n+1}$.
This map may modify the intersection $X \cap H^{(n)}$,
so it is not necessarily a point modification.
Also, we will not consider any modification
of $Y^{(d+1)}$ for $d < n$ induced by this map,
as in the case of $\triangle_A$ above.
For any $J \subset \hat I$, we will denote by
\[
  \hat D_J^\circ \subset \hat Y_A
\]
the corresponding stratum.

Now, given any finite subset $A \subset \Q_{>0}$,
we can choose a regular subdivision of
$\triangle_A$ of $\triangle$ which, for any $\alpha \in A$,
and $\sigma_J \in \hat\triangle$, refines the cone
\begin{equation} \label{eq:alpha_cone}
  \set{v\in \sigma_J}{\alpha u_{J,g}(v) =  u_{J,f}(v_i) }
  \subset \sigma_J.
\end{equation}
Let $Y_A$ be the induced space.
Let $\hat\triangle_A$ be the subdivision of $\hat \triangle$,
obtained as follows:
\begin{itemize}
\item
The cone $\sigma$ generated by $i_X$ and $i_H$ is subdivided so that each cone
\cref{eq:alpha_cone} is an element of this subdivision.
\item
If $i_X\notin J$ and $i_H \notin J$, then the cone $\sigma_J$
is subdivided as in $\triangle_A$. 
\item
Any cone generated by $i_X, i_H$ and another nonempty set $J$
not containing $i_X$ or $i_H$ is subdivided as the join of the two
subdivisions above, which induces a subdivision of cones generated by
$i_X$ or $i_H$ and $J$.
\end{itemize}

The morphism $Y_A \to \C^{n+1}$ is then a point modification, and we
use the notation in \cref{block:emb_res}, with $D_A = \cup_{i\in I_A} D_i$.
For each exceptional divisor $D_i$, with $i \in I_A$, denote by
$\hat D_i$ the corresponding divisor in $\hat Y_A$, i.e.
the strict transform of $D_i$ in $\hat Y_A$, and set
\[
  \hat D_i^\circ
  =
  \hat D_i \setminus \bigcup_{j\in I_A \setminus \{i\}} \hat D_j
\]
Since the map $\hat D_i \to D_i$ only modifies strata of codimension
$\geq 2$, we see that the map
\begin{equation} \label{eq:hatD_iso}
  \hat D_i^\circ \setminus (\hat X_A \cup \hat H_A)
  \stackrel{\cong}{\to}
  D_i^\circ \setminus (\tilde X_A \cup \tilde H_A)
\end{equation}
is an isomorphism. Here $\hat X_A, \hat H_A \subset \hat Y_A$
and $\tilde X_A, \tilde H_A \subset Y_A$ are the strict transforms
of $X,H$.
\end{block}

\begin{lemma} \label{lem:Asep}
Assume the notation introduced in \cref{block:Asep}.
\begin{blist}
\item \label{it:Asep_Asep}
The composed map $\pi_A:Y\to \C^{n+1}$ is an $A$-separating
embedded resolution which
factors through the blow-up of $\C^{n+1}$ at $0$.
\item \label{it:Asep_indet}
For any $\alpha \in A$, the map $\hat\pi$ resolves the indeterminacy
locus of $\phi_\alpha$. In other words, pulling back via $\hat\pi$
gives a genuine map
\[
  \hat \pi^*_A \phi_\alpha:\hat Y_A \to \CP^1.
\]
\item \label{it:Asep_gen}
If the coordinates $x_0,x_1,\ldots,x_n$ are generic with respect to $\pi_0$,
then they are generic with respect to $\pi$.
\item \label{it:Asep_chi}
For every $i \in I_A$, we have
\[
  \chi(D_{i,A}^\circ \setminus \tX)
  =
\begin{cases}
  \chi(D^\circ_{i,0} \setminus \tX_0) & i \in I,\\
  0 & i \notin I.
\end{cases}
\]
\end{blist}
\end{lemma}

\begin{proof}
\cref{it:Asep_Asep} and \cref{it:Asep_indet} follow from the requirement
\cref{eq:alpha_cone}, which guarantees that the irreducible component
of the zero set and the poles of $\phi_\alpha$ do not intersect, with
the possible exception of $\tilde X_A$ and $\tilde H_A$ in $Y_A$.
We have already seen \cref{it:Asep_gen} and \cref{it:Asep_chi}.
\end{proof}

\section{The Jacobian and alternating Jacobian polygons} \label{s:jacob}

\begin{block}
In this section we introduce and discuss an invariant of an isolated
hypersurface singularity,
defined by Teissier \cite{Tei,teissier_jacobian}.
We give a formula for the Jacobian polygon
in terms of an embedded resolution,
which induces a natural splitting into terms which can be seen
as the integral of the Hironaka function over the Milnor fiber at
radius zero with respect to the Euler characteristic.
These terms are the \emph{alternating Jacobian polygons}.

In \cite{Mich_Jac}, Michel proves a formula similar to
\cref{eq:ttps} in the context of a finite morphism
$(f,g):(X,p) \to (\C^2,0)$,
where $X$ is a surface with an isolated singularity at $p$.
This statement reduces to our formula with $n=1$
if $(X,p) = (\C^2,0)$ is smooth, and
$g$ is a generic linear function.

We will assume that $\pi:Y \to \C^{n+1}$ is an embedded resolution of
$f$, and that $x_0,x_1,\ldots,x_n$ are generic coordinates with respect
to $f$ and $\pi$.
\end{block}

\begin{definition} \label{def:m_n_alpha}
With the polar curve $P = P_1 \cup \cdots \cup P_l$ with respect
to $f$ and $g = x_n$, set
\[
  m_q = \left(X,P_q\right)_{\C^{n+1},0},\qquad
  n_q = \left(H,P_q\right)_{\C^{n+1},0},\qquad
  \alpha_q = \frac{m_q}{n_q}.
\]
For any $d = 0,1,\ldots,n$ and $q = 1,2,\ldots,l^{(d+1)}$, define similarly
\[
  m^{(d+1)}_q = \left(X^{(d+1)},P^{(d+1)}_q\right)_{H^{(d+1)},0},\qquad
  n^{(d+1)}_q = \left(H^{(d)},P^{(d+1)}_q\right)_{H^{(d+1)},0},\qquad
  \alpha^{(d+1)}_q = \frac{m^{(d+1)}_q}{n^{(d+1)}_q}.
\]
\end{definition}

\begin{definition} \label{def:TN}
Let $x_0,x_1,\ldots,x_n$ be generic coordinates in $\C^{n+1}$ with respect
to $f$.
\begin{blist}
\item \label{it:TN_J}
The Jacobian polygon associated with $(X,0)$ is
\begin{equation} \label{eq:TPS}
  \TPS(f,0)
  =
  \sum_{q=1}^l
  \newpol{m_q}{n_q}
  \in \new.
\end{equation}
\item \label{it:TN_Jd}
For $0 \leq d \leq n$, we define
\[
  \TPS^{(d+1)}(f,0) = \TPS\left(f|_{H^{(d+1)}},0\right)
  =
  \sum_{q=1}^{l^{(d+1)}}
  \newpol{m^{(d+1)}_q}{n^{(d+1)}_q}
  \in \new.
\]
\item \label{it:TN_AJ}
We define the \emph{alternating Jacobian polygon of $f$} as
\begin{equation} \label{eq:tTN}
  \tTPS(f,0)
  =
  \sum_{d=0}^n (-1)^{n-d} \TPS^{(d+1)}(f,0)
  \in \Knew.
\end{equation}

\item \label{it:TN_AJd}
For $d=0,1,\ldots,n$, set
\[
  \tTPS^{(d+1)}(f,0)
  =
  \tTPS(f|_{H^{(d+1)}},0) \in K\new
\]
and $\tTPS(f,0)^{(0)} = 0$.
\end{blist}
\end{definition}

\begin{rem} \label{rem:tps_props}
\begin{blist}

\item
The set of coordinates generic with respect to $f$ forms an open
subset in the set of linear maps $\C^{n+1} \to \C^{n+1}$,
over which we have a deformation for each $d=0,1,\ldots,n$ with
fibers $X^{(d+1)}$, and these deformations are
(c)-equisingular (see \cite{Tei} for the definition of (c)-equisingularity).
As a result, the Jacobian polygons
are independent of the choice of generic coordinates.

\item \label{it:n0}
In the case $n=0$, there exists a unit $u \in \C\{x_0\}$ such that
$f(x_0) = x_0^e\cdot u(x_0)$, where $e$ is the multiplicity of $X$ at $0$.
In this case, $l = 1$, $(X,P_1)_0 = e$, $(H,P_1)_0 = 1$ and so
\[
  \tTPS(f,0)
  = \TPS(f,0)
  = \newpol{e}{1}.
\]
As a result, if $e$ is the multiplicity of the hypersurface $X$ at $0$, then 
\[
  \TPS^{(1)}(f,0) = \newpol{e}{1}.
\]

\item \label{it:two_summands}
It follows immediately from \cref{eq:tTN} that
\begin{equation} \label{eq:J_AJ_AJ}
  \TPS^{(d+1)}(f,0)
  =
  \tTPS^{(d+1)}(f,0) + \tTPS^{(d)}(f,0).
\end{equation}

\item \label{it:tps_length}
As a result of \cite[1.4 Remarque]{Tei}, we have
\[
  \ell\left(\TPS^{(d+1)}(f,0)\right) = \mu^{(d+1)} + \mu^{(d)},\qquad
  h\left(\TPS^{(d+1)}(f,0)\right) = \mu^{(d)}
\]
for $d=0,1,\ldots,n$. Note that here, by convention, we set $\mu^{(0)} = 1$.

\item \label{it:tps_telescope}
By the previous remark, we get a telescopic series
\[
  \ell(\tTPS(f,0))
  = \sum_{d=0}^n (-1)^{n-d} \left(\mu^{(d+1)} + \mu^{(d)}\right)
  = \mu^{(n+1)} + (-1)^n.
\]

\item \label{it:AJ0}
As a consequence of the previous point,
we have $\ell(\tTPS(f,0)) \neq 0$ unless $n$ is odd and $\mu = 1$, which
happens precisely when $(X,0)$ is a Morse point, i.e. when the Hessian
matrix of $f$ at $0$ is invertible. In this case one readily verifies
that $(X^{(d+1)},0)$ is a Morse point for all $d$, and that
\[
  \TPS^{(d+1)}(f,0) = \newpol{2}{1}, \qquad d=0,1,\ldots,n,
\]
and we find $\tTPS(f,0) = 0$ in this case. We have proved that the
following are equivalent
\begin{itemize}
\item
$\tTPS(f,0) = 0$,
\item
$\ell\left(\tTPS(f,0)\right) = 0$,
\item
$n$ is odd and $f$ is Morse.
\end{itemize}

\item \label{it:tps_loj}
By \cite[1.7 Corollaire 2]{Tei}, the Jacobian polygon
contains the \L ojasiewicz exponent:
\begin{equation} \label{eq:loj_deg}
  \Loj(f,0)
  = \deg \TPS(f,0) - 1.
\end{equation}

\item \label{it:alpha_mult}
If $\alpha \in \supp(\TPS(f,0))$, then $\alpha \geq \mult(X,0)$. Indeed,
we have
\begin{equation} \label{eq:alpha_frac}
  \alpha = \frac{(X,P_q)_0}{(H,P_q)_0}
\end{equation}
for some $q=1,2,\ldots,l$, and
by \cite[1.2 Th\'eor\`eme 1]{Tei}, we have $(H,P_q)_0 = \mult(P_q,0)$
for each $q=1,\ldots,l$. Therefore,
\[
  \alpha
  = \frac{(X,P_q)_0}{(H,P_q)_0}
  \geq \frac{\mult(X,0) \cdot \mult(P_q,0)}{\mult(P_q,0)}
  = \mult(X,0).
\]

\end{blist}
\end{rem}

\begin{thm} \label{thm:tps}
Let $\pi:Y\to \C^{n+1}$ be an embedded resolution of $f$ which factors
through the first blow-up. Then
\begin{equation} \label{eq:tps}
  \TPS(f,0)
  =
  (-1)^n \sum_{i \in I}
                \chi(D_i^\circ \setminus (\tX\cup\tH))
                \newpol{m_i}{n_i}
\end{equation}
and
\begin{equation} \label{eq:ttps}
  \tTPS(f,0)
  =
  (-1)^n \sum_{i \in I}
                \chi(D_i^\circ \setminus \tX)
                \newpol{m_i}{n_i}.
\end{equation}
\end{thm}

\begin{proof}
Using \cref{eq:TPS,eq:tTN} and additivity of the Euler characteristic,
we see that \cref{eq:ttps} follows from \cref{eq:tps}, so we only prove
\cref{eq:tps}.
For any rational number $\alpha = m_\alpha/n_\alpha$, with
$m_\alpha,n_\alpha > 0$ and $\gcd(m_\alpha, n_\alpha) = 1$, set
\[
  I_\alpha = \set{i\in I}{\frac{m_i}{n_i} = \alpha}.
\]
We have to show that
\begin{equation} \label{eq:t_alpha}
  \sum_{\alpha_q = \alpha}
  \left(X,P_q \right)_{\C^{n+1},0}
  = 
  (-1)^n
  \sum_{i \in I_\alpha} m_i \chi(D_i^\circ \setminus (\tX\cup\tH))
\end{equation}
for all $\alpha$.
Indeed, on the left hand side, we have $m_\alpha$ times
the coefficient in front of $\{\alpha\}$ in $\TPS(f,0)$,
whereas on the right hand side we have $m_\alpha$ times
the coefficient of $\{\alpha\}$ on the right hand side of \cref{eq:tps}.
Note that in \cref{eq:tps}, each side is a finite sum.
Therefore we know, a priori, that \cref{eq:t_alpha} holds for all but
a finite list of rational numbers, those of the form
$\alpha = m_i/n_i$ for an exceptional divisor $D_i$, $i \in I$
or $\alpha = m_q / n_q$ for a polar branch $P_q$.

For any reduced fraction $\alpha = m_\alpha/n_\alpha$ as before,
define the function
\[
  \phi_\alpha = \frac{g^{m_\alpha}}{f^{n_\alpha}}.
\]
Fixing an $\alpha$, as well as a polar branch $P_q$,
the order of the pullback of $\phi_\alpha$ to the normalization of $P_q$
is a positive multiple of $n_q m_\alpha - m_q n_\alpha$, and so,
\[
  \lim_{x\to 0} \phi_\alpha |_{P_q}
\begin{cases}
  = \infty & \alpha_q > \alpha, \\
  \in \C^* & \alpha_q = \alpha, \\
  = 0      & \alpha_q < \alpha. \\
\end{cases}
\]
As a result, if
\[
  M = M_{\epsilon, \eta}
  = B_\epsilon^{2n+2} \cap f^{-1}(D_\eta)
\]
is a small Milnor tube, i.e. $0\mleq\eta<\epsilon\mleq 1$,
then we can find positive real constants $a,b$
satisfying the following conditions:
\begin{enumerate}
\item
For any branch $P_q$ of the polar curve, we have
\[
  a < \left|\phi_{\alpha_q}|_{P_q \cap M}\right| < b.
\]
\item
For any two branches $P_q, P_r$ of the polar curves such that
$\alpha_r < \alpha_q$, we have
\[
  \left|\phi_{\alpha_q}|_{P_r \cap M}\right|
  <
  a \qquad
  \mathrm{and} \qquad
  b
  <
  \left|\phi_{\alpha_r}|_{P_q \cap M}\right|.
\]
\item
If $\emptyset \neq J \subset I_\alpha$, for some $\alpha$,
and $c$ is a nonregular value of the restriction
$\pi^*\phi_\alpha|_{D_J^\circ}$, then $a < c < b$.
\end{enumerate}
Set
\[
  U = \set{y\in\C}{a \leq |y| \leq b}.
\]
With small $0 < \eta \mleq \epsilon$, denote in this proof
\[
  F = F_{\alpha, \epsilon, z}
  =
  f^{-1}(z) \cap B_\epsilon^{2n+2} \cap \phi_\alpha^{-1}(U).
\]
Then, for $0 < |z| < \eta$, we have
\begin{equation} \label{eq:XPq_chiF}
  \sum_{\alpha_q = \alpha} (X,P_q)_{\C^{n+1},0}
  =
  (-1)^n \chi(F_{\alpha, \epsilon, z})
\end{equation}
Indeed,
the restriction $\phi_\alpha|_F$ of $\phi_\alpha$ to
$F = F_{\alpha,\epsilon,z}$ coincides with that of
$g^{m_\alpha} / |z|^{n_\alpha}$.
Therefore, $\phi_\alpha|_F$ is a submersion everywhere, except for
at intersection points $F \cap P$. It follows from the construction that
the branch $P_q$ intersects $F$ in exactly $(X,P_q)_0$ points
if $\alpha_q = \alpha$, and in no points otherwise.
As a result, the restriction of $\phi_\alpha$ to a map
$F \to U$ is a
submersion, except for at precisely $\sum_{\alpha_q = \alpha} (X,P_q)_0$ points,
where it has an $A_1$ singularity. If $G$ is the fiber over a general point
of this map, then we find
\[
  \chi(F) = \chi(G)\chi(U) + (-1)^n \sum_{\alpha_q = \alpha} (X,P_q)
\]
proving \cref{eq:XPq_chiF}, since $\chi(U) = 0$.

Now, with $\alpha = m/n$ fixed, consider the subset
$A = \{\alpha\} \subset \Q$, and
let $\pi_A$ and $\hat\pi_A$ be as in \cref{block:Asep}.
By \cref{lem:Asep}, it suffices to prove the lemma for
$\pi_A$, so in order to save effort on notation, let us assume that
$\pi = \pi_A$.
We identify $F_{\alpha,\epsilon,z}$ with its preimage in
$\hat Y = \hat Y_A$ by $\hat \pi$.
Choose any metric on $\hat Y$, and consider a gradient-like vector
field for the function $|\hat \pi^* f|^2$
which is tangent to the preimages
$\phi_\alpha^{-1}(a)$ and $\phi_\alpha^{-1}(b)$.
Note that by construction, $a$ and $b$ are regular values of $\phi_\alpha$.
Integrating this vector field induces a map
\[
  F_{\alpha,\epsilon,\eta} \to \hat D \cap \phi_\alpha^{-1}(U),
\]
whose fiber over a point in $\hat D^\circ_J \cap \phi_\alpha^{-1}(U)$
consists of
$m_J = \gcd\set{m_j}{j\in J}$ disjoint copies of $(S^1)^{|J|-1}$.
As a result,
\begin{equation} \label{eq:chiF_mchiD}
  \chi\left(F_{\alpha,\epsilon,\eta}\right)
  =
  \sum_{\emptyset \neq J \subset I_\alpha}
          \chi\left(\hat D^\circ_J \cap \phi^{-1}_\alpha(U)\right)
    \cdot m_J
    \cdot \chi\left((S^1)^{|J|-1}\right)
  =
  \sum_{ i \in I_\alpha }
   m_i \chi\left(\hat D_i^\circ \cap \phi^{-1}_\alpha(U)\right)
\end{equation}
Now, by using a gradient-like vector field for $\pi^*\phi_\alpha|_{D_i^\circ}$
for $i\in J_\alpha$, satisfying $m_i/n_i = \alpha$,
and $\hat D_i \cong D_i$, we find that the natural homotopy equivalence
\begin{equation} \label{eq:last}
  \hat D_i^\circ \cap \pi_\alpha^*\phi^{-1}_\alpha(U)
  \hookrightarrow
  D_i^\circ \setminus \left(\tilde X \cup \tilde H\right)
\end{equation}
In particular, the two spaces have the same Euler characteristic.
As a result, \cref{eq:t_alpha} follows from
\cref{eq:XPq_chiF} and \cref{eq:chiF_mchiD}.
\end{proof}

\section{The vanishing rate filtration} \label{s:vanishing}

\begin{block}
In this section we describe a Morse function and a gradient-like vector
field for it on the Milnor fiber. The critical points of this Morse
function are intersection
points with polar curves in subspaces of varying dimension.
As a result, we have a corresponding Morse--Smale complex
$(C_\cdot,\partial_\cdot)$ which computes the homology of 
the Milnor fiber, as in \cite[Theorem 7.4]{hMilnor}.
This construction can be made in such
a way that all but $\mu = \mu^{(n+1)}$ handles of index $n$
combine to form a ball $B^{(2n)}$, as in \cite{le_perron}.
Now, to each polar curve, we associate
its \emph{vanishing rate}, following Teissier \cite[3.6.4]{intro_equising}.
This induces a grading on the Abelian group
$C_\cdot$, which induces
an increasing filtration of the complex $(C_\cdot, \partial_\cdot)$, which
we call the \emph{vanishing rate filtration}.
\end{block}
\begin{block}
In order to describe the Milnor fiber as a handle body, we construct
a Morse function $\psi^{(d+1)}_{\epsilon,\eta}$
on the Milnor fiber $F^{(d+1)}_{\epsilon,\eta}$, as well as
a gradient-like vector field $\xi^{(d+1)}_{\epsilon,\eta}$ for it.
As a result, we obtain the Morse--Smale complex generated by the
critical values
of $\psi_{\epsilon,\eta} = \psi^{(n+1)}_{\epsilon,\eta}$, whose
differential counts trajectories of
$\xi_{\epsilon,\eta} = \xi^{(d+1)}_{\epsilon,\eta}$.
We will skip the indices $\epsilon,\eta$ when they are clear from context.

First, we construct a vector field $\zeta$ on $Y^{(d+1)}$.
For any point $x \in Y^{(d)} \cap D^{(d+1)}$, let $J \subset I^{(d+1)}$
be the set such that $x \in D^\circ_J$.
If $x \notin \tX$, then there exist holomorphic coordinates
$u_0,u_1,\ldots,u_d$ centered at $x$ in a chart $U_x \ni x$ 
with the property that for each $i \in J$ there is a
$k_i \in \{1,2,\ldots,|J|\}$ such that $D_i \cap U_x = \{u_{k_i} = 0\}$
and
\[
  \pi^* f|_{U_x} = \prod_{i\in J} u_{k_i}^{m_i},\qquad
  \pi^* g^{(d+1)}|_{U_x} = u_0 \prod_{i\in J} u_{k_i}^{n_i}.
\]
If $x \in \tX$, then we can find a similar coordinate chart
such that
\[
  \pi^* f|_{U_x} = u_1 \prod_{i\in J} u_{k_i}^{m_i},\qquad
  \pi^* g^{(d+1)}|_{U_x} = u_0 \prod_{i\in J} u_{k_i}^{n_i}.
\]
In either case, define a vector field in $U_x$
\[
  \zeta_x = u_0 \frac{\partial}{\partial u_0}
\]
By compactness, we can cover $Y^{(d)} \cap D^{(d+1)} = D^{(d)}$
with finitely many such
$U_x$, say $H \cap D^{(d)} \subset \cup_{x \in T} U_x$, with $|T| < \infty$.
We will assume that $\epsilon$ is chosen small enough that
$\pi^{-1}(B_\epsilon \cap H^{(d)}) \subset \cup_{x \in T} U_x$.
In fact, if we fix some Riemannian metric on $Y^{(d+1)}$, inducing a
metric $d$, then,
if $\kappa > 0$ is small enough, then the set
\[
  N_\kappa
  = \set{x \in Y^{(d+1)} \cap \pi^{-1}(B_\epsilon)}
                 { d(x,Y^{(d)}) < \kappa }
\]
is a tubular neighborhood of $Y^{(d)} \cap \pi^{-1}(B_\epsilon)$, and we
have $\overline{N}_\kappa \subset \cup_{x\in T} U_x$ for $\kappa$ small enough.
Now, set
\[
  U_0
  = \pi^{-1}(B_\epsilon)
    \cap Y^{(d+1)}
    \setminus \overline{N}_\kappa.
\]
We have the vector field
\[
  \zeta_0^{(d+1)}(x)
  = \nabla \left|\pi^*\left( g^{(d+1)}|_{F_{\epsilon,z}} \right) \right|^2 (x).
\]
for $x \in U_0$, where $z = \pi^*f(x)$.
We define the vector field $\zeta$ on $Y^{(d+1)} \cap \pi^{-1}(B_\epsilon)$
by gluing together $\zeta_0$ and the $\zeta_x$ for $x\in T$ via
a partition of unity.
Observe that
\begin{itemize}

\item
The vector field $\zeta$ is tangent to the Milnor fibers
$F^{(d+1)}_{\epsilon,z}$
(which we identify with their preimages in $Y^{(d+1)}$).

\item
For any $z \in D^*_\eta$,
the vector field $\zeta_{F^{(d+1)}_{\epsilon,z}}$
has nondegenerate singular points at intersection points
$P^{(d+1)} \cap F^{(d+1)}_{\epsilon,z}$ of index $d$.

\item
$\zeta$ is defined in a neighborhood around $Y^{(d)}$, vanishes along
$Y^{(d)}$ in a nondegenerate way. In fact, the Hessian of
$\zeta$ restricted to the normal bundle of $Y^{(d)} \subset Y^{(d+1)}$
is positive definite.

\end{itemize}

As a result of the last item, we can, and will, assume that $\kappa$
is chosen small enough so that there exists a trivialization
\[
  \iota^{(d+1)}: Y^{(d)} \times D_\kappa \cong N_\kappa^{(d+1)}
\]
of $N_\kappa$, with the property that for any $x \in Y^{(d)}$,
the set $\iota^{(d+1)}(\{x\}\times D_\kappa)$ is the unstable
manifold of $\zeta$ at $x$.

Next, for a fixed $z$, we construct a Morse function
\[
  \psi_{\epsilon,z}^{(d+1)} : F^{(d+1)}_{\epsilon,z} \to \R
\]
and a gradient-like vector field $\xi_{\epsilon,z}^{(d+1)}$ recursively
for $d = 0,1,\ldots,n$. Starting with $d = 0$, we set
\[
  \psi^{(1)} = 0,\qquad
  \xi^{(1)} = 0.
\]
Assuming that we have defined $\psi_{\epsilon,\eta}^{(d)}$ and
$\xi_{\epsilon,\eta}^{(d)}$ for
some $d > 0$, we then construct $\psi^{(d+1)}$ and $\xi^{(d+1)}$.
Take a $\Cinf$ function $B:\R\to\R$ satisfying $B|_{(-\infty,1/3]} = 1$
and $B|_{[2/3,\infty)} = 0$, and set
\[
  \tilde \psi^{(d)}(\tilde x)
  =
\begin{cases}
  B\left(|g^{(d+1)}(x)|^2/\rho\right)
    \psi^{(d)}(x') &
    \mathrm{if}\; x = \iota(x',t), \\
  0                      & \mathrm{else.} \\
\end{cases}
\]
where $\rho > 0$ is so small that if $x \in F^{(d+1)}_{\epsilon,z}$
and $|g^{(d+1)}(x)| < \rho$ then $x \in N^{(d+1)}_\kappa$.
Then, the function
\[
  \psi^{(d+1)}_{\epsilon,\eta}
  = |g^{(n+1)}|^2 + \nu \tilde \psi^{(d)}
\]
is Morse on $F^{(d+1)}_{\epsilon,\eta}$ for $\nu$ small enough.
In fact, since $g^{(d+1)}|_{F^{(d+1)}_{\epsilon,\eta}}$ has $A_1$
singularities at intersection points with $P^{(d+1)}$,
$\psi^{(d+1)}$ has nondegenerate singularities there of index $d$.
Furthermore, the Hessian of $|g^{(d+1)}|^2$ restricted to 
the normal bundle of $F^{(d)}$ in $F^{(d+1)}$ is positive definite.
Therefore, we see by induction that the index of $\psi^{(n+1)}$ at an
intersection point in $F^{(n+1)} \cap P^{(d+1)}$ has index $d$.

By extending $\xi$ similarly as above, using the product structure
on the image of $\iota^{(d+1)}$ to a vector field $\tilde\xi^{(d)}$
on $F^{(d+1)}$ with support near $F^{(d)}$ we set
\[
  \xi^{(d+1)} = \nabla |g^{(d+1)}|^2 + \nu \tilde \xi^{(d)}.
\]
Then for $\nu$ small enough, $\xi^{(d+1)}$ is gradient-like for
$\psi^{(d+1)}$.
\end{block}

\begin{lemma} \label{lem:trajs}
Fix polar curve branches $P^{(d+1)}_q$ and $P_r^{(d+1)}$ for some $d$.
Then, if, for $\eta$ arbitrarily small, there exists a $z \in D^*_\eta$
and a trajectory
$\gamma:\R \to F^{(d+1)}_{\epsilon,z}$ of $\xi^{(d+1)}$ satisfying
\[
  \lim_{t\to-\infty} \gamma(t) \in P^{(d)}_r,\qquad
  \lim_{t\to+\infty} \gamma(t) \in P^{(d+1)}_q,
\]
then
\begin{equation} \label{eq:traj_ineq}
  \alpha^{(d)}_r
  \leq
  \alpha^{(d+1)}_s.
\end{equation}
\end{lemma}
\begin{proof}
We start by making some assumptions on the resolution $\pi$.
Similarly as in the proof of \cref{thm:tps}, let
$\triangle$ be the conical polyhedral complex associated with the
divisor
\[
  D \cup \tX \subset Y.
\]
Note that we do not include the divisor $Y^{(n)} \subset Y$. Set
\[
  \alpha = \alpha_q^{(d+1)},\qquad
  \beta = \alpha_r^{(d)}, \qquad
  A = \{\alpha, \beta\}.
\]
Assume, then, that $\pi = \pi_A$ is an embedded resolution of $f$
which satisfies the conclusion of \cref{lem:Asep}.
In particular,
the indeterminacy locus of the meromorphic function
\[
  \pi^* \phi_\alpha^{(d+1)}
  = \frac{\left(g^{(d+1)}\right)^m}
         {\left(f^{(d+1)}\right)^n},\qquad
  \alpha = \frac{m}{n}
\]
is contained in $Y^{(d)}$. Similarly,
the indeterminacy locus of the meromorphic function
\[
  \pi^* \phi_\alpha^{(d)}
  = \frac{\left(g^{(d)}\right)^m}
         {\left(f^{(d)}\right)^n},\qquad
  \alpha = \frac{m}{n}
\]
is contained in $Y^{(d-1)}$.

Let us assume that $z_k$ is a sequence of small numbers $z_k \in \C$
converging to zero
such that there exist trajectories $\gamma_k$ in $F^{(d+1)}_{z_k,\epsilon}$
as described in the statement of the lemma.
We then want to prove that $\beta \leq \alpha$.
For each $k$, the vector field $\xi^{(d+1)}_{\epsilon,z_k}$ points
outwards along the boundary of
$N^{(d+1)}_\kappa \cap F^{(d+1)}_{\epsilon,z_k}$, and so
the trajectory $\gamma_k$ intersects this boundary in exactly one point.
By reparametrization, we will assume that this point, say $p_k$,
is the value of $\gamma_k$ at $0$.

\begin{figure}[ht]
\begin{center}
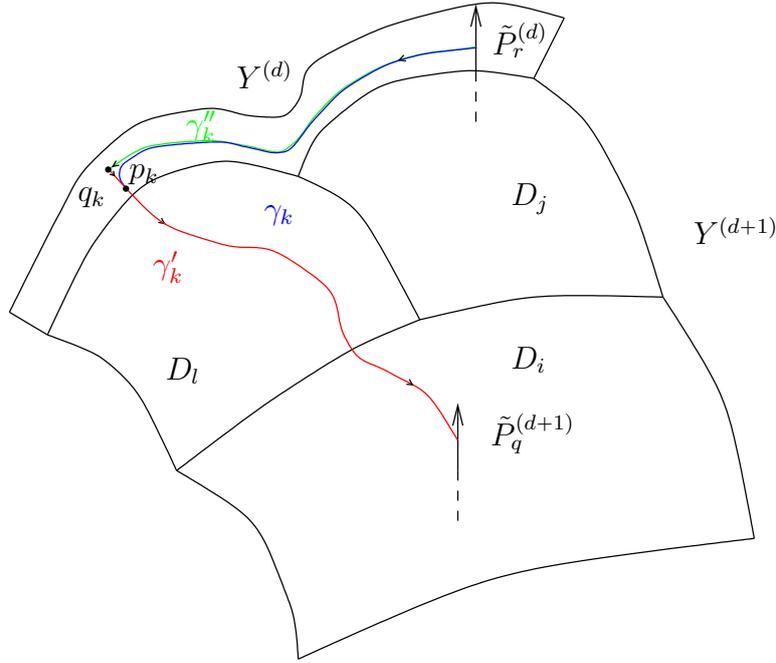
\caption{The trajectories $\gamma_k, \gamma_k', \gamma_k''$ in $Y^{(d+1)}$.}
\label{fig:traj}
\end{center}
\end{figure}

The restriction of $\gamma_k$ to the positive real axis is then a trajectory
of $\zeta^{(d+1)}$ as well. Denote by $\gamma'_k$ the trajectory of
$\zeta^{(d+1)}$ which coincides with $\gamma_k$ along $\R_{\geq 0}$.
We then have
$\gamma_k(0) = \iota^{(d+1)}(q_k,\theta_k)$ for some $q_k \in Y^{(d+1)}$
and $\theta_k \in \partial D_\kappa$.
The continuation of $\gamma'_k$ to the negative axis then
follows the segment from $(q_k,\theta_k)$ towards $(q_k,0)$, and we have
\[
  \lim_{t\to+\infty} \gamma'_k(t) = q_k \in Y^{(d)}.
\]
Since $\phi_\alpha(\gamma'(t))$ is increasing, and by choosing
$a,b \in \R_{>0}$ as in the proof of \cref{thm:tps}, we find
\[
  |\pi^*\phi^{(d+1)}_\alpha(p_k)|
  =
  |\pi^*\phi^{(d+1)}_\alpha(\gamma_k(0))|
  <
  \lim_{t\to+\infty} |\pi^*\phi^{(d+1)}_\alpha(\gamma_k(t))|
  \leq
  b
\]
By compactness, the sequence $p_k$ has an accumulation point
$p \in Y^{(d+1)}$. By taking a subsequence, we will assume that,
in fact, $p_k \to p$ when $k \to \infty$.
By continuity of $\pi^*\phi^{(d+1)}_\alpha$ away from $Y^{(d)}$, we then have
\[
  |\pi^*\phi_\alpha(p)| \leq b\qquad \mathrm{and}\qquad
  \pi^*f(p) = 0.
\]
In particular, we have $p \in D^{(d+1)} \cup \tX^{(d+1)}$.
Furthermore, $p \notin \tX^{(d+1)}$ and $p \notin D_i$ for any
$i\in I^{(d+1)}$ satisfying $m_i / n_i > \alpha$.
As a result, as for any point on $D^{(d+1)}$,
there exists a unique $J \subset I^{(d+1)}$ such that $q \in D^\circ_J$,
but we also have $m_i / n_i \leq \alpha$ for any $i \in J$.

Now, since $p_k$ converges to $p$, there exists a $q \in Y^{(d)}$
such that $q_k \to q$. Since the vector field $\zeta$ is tangent
to any $D_i$ for $i \in I^{(d+1)}$, we find that
\[
  q \in D_J^\circ.
\]
Since $q \in Y^{(d)}$,
there exists a unique $K \subset I^{(d)}$ such that
$q \in D_K^\circ$, and the map $\omega^{(d+1)}:I^{(d)} \to I^{(d+1)}$
induces a bijection $K \to J$.

Next, we restrict our attention to the part of $\gamma_k$ contained in
$N_k^{(d+1)}$, which has a product structure via $\iota^{(d+1)}$. For
$t<0$, write
\[
  \gamma_k(t) = \iota^{(d+1)}( \gamma''_k(t), \vartheta(t) ).
\]
By construction, $\gamma''_k(t)$ is then a reparametrization of
a trajectory of $\xi^{(d)}_{\eta,z_k}$, and $\vartheta$ is a parametrization
of the open segment from $0 \in D_\eta$ to $\theta_k \in \partial D_\kappa$.
Similarly as above, we have for any $k$
\[
  a
  \leq
  \lim_{t\to -\infty}
    \left|\pi^*\phi^{(d)}_\beta\left( \gamma''(t)\right)\right|
  \leq 
  \left|\pi^*\phi^{(d)}_\beta \left(q_k\right)\right|
\]
and so $a \leq \phi_\beta^{(d)}(q)$. It follows that for any
$i \in K \subset I^{(d)}$ we have $\beta \leq m_i/n_i$.
Therefore,
\[
  \beta
  \leq
  \frac{m_i}{n_i}
  =
  \frac{m_{\omega^{(d+1)}(i)}}{n_{\omega^{(d+1)}(i)}}
  \leq
  \alpha.\qedhere
\]
\end{proof}

\begin{block} \label{block:morse_smale}
Denote by
\[
  \Sigma^{(d+1)} = P^{(d+1)} \cap F_{\epsilon, \delta}
\]
the set of critical points of $\psi^{(d+1)}$ having index $d$,
which is partitioned by setting
$\Sigma^{(d+1)}_q = \Sigma^{(d+1)}\cap P^{(d+1)}_q$ for any $q$.
Let $(C_\cdot,\partial_\cdot)$ be the Morse--Smale complex associated
with $\psi = \psi^{(n+1)}$ and $\xi = \xi^{(n+1)}$ so that
\[
  C_d = \Z\langle \Sigma^{(d+1)} \rangle,\qquad
  C_\cdot = \bigoplus_{d=0}^n C_d.
\]
Thus, if $x \in \Sigma^{(d+1)}$ and $y \in \Sigma^{(d)}$, then
the coefficient in front of $y$ in $\partial_d(x)$ is a signed
count of trajectories $\gamma$ of $\xi^{(n+1)}$ (or a small
perturbation of $\xi^{(d+1)}$) satisfying
\[
  \lim_{t\to-\infty} \gamma(t) = y,\qquad
  \lim_{t\to+\infty} \gamma(t) = x.
\]
If $p \in \Sigma_q^{(d+1)}$, the we
define the \emph{vanishing rate} of $p$ as
$\alpha(p) = \alpha_q^{(d+1)}$.
The \emph{vanishing rate filtration} on $C_\cdot$
is the filtration $F_\cdot$ given by
\[
\begin{split}
  F_\alpha C_d
  &= \Z\gen{p \in F^{(n+1)} \cap P^{(d+1)}}{\alpha(p) < \alpha}
  \subset C_d, \\
  F_\alpha C_\cdot
  &= \sum_{d=0}^n F_\alpha C_d
  \subset C
\end{split}
\]
for any $\alpha \in \Q$.
For $d = 0,1,\ldots,n$, we define $(C_\cdot^{(d+1)}, \partial^{(d+1)}_\cdot)$
similarly as the Morse--Smale complex associated with
$\psi^{(d+1)}$ and $\xi^{(d+1)}$. The Abelian groups $C_\cdot^{(d+1)}$
are filtered in the same way by the vanishing rates.
By the construction of the Morse--Smale complex,
\cref{lem:trajs} immediately implies
\end{block}

\begin{cor} \label{cor:filtration}
The filtration $F_\cdot$ is a filtration of $(C_\cdot,\partial_\cdot)$
as a complex, that is,
\[
\pushQED{\qed}
  \partial_d(F_\alpha C_d) \subset F_\alpha C_{d-1},\qquad \alpha \in \Q.
  \qedhere
\popQED
\]
\end{cor}

\begin{comment}
\begin{prop}
We have
\[
\pushQED{\qed}
  \TPS(f,0)
  = \sum_{\alpha \in \Q} \rk\Gr_\alpha^F C_n \cdot [\alpha],\qquad
  \tTPS(f,0)
  = (-1)^n \sum_{d=0}^n (-1)^d
           \sum_{\alpha \in \Q} \rk\Gr_\alpha^F C_d \cdot [\alpha]. \qedhere
\popQED
\]
\end{prop}
\end{comment}

\section{Nonnegativity for the alternating Jacobian polygon} \label{s:nonneg}

\begin{block}
In this section, we prove \cref{thm:mono}, as well as several
corollaries. This result shows that, although the alternating
Jacobian polygon does not always have nonnegative coefficients
like Teissier's Jacobian polygon, the length of
any of its truncations is nonnegative.

The proof of this result uses the filtration on the Morse--Smale complex
introduced in the previous section. This argument is sketched out
informally in \cref{block:key}.
\end{block}

\begin{thm} \label{thm:mono}
For any $\alpha \in \Q_{>0}$ and $d = 1,\ldots,n$, we have
\begin{equation} \label{eq:hlh}
  h\left( \TPS^{(d+1)}_{\geq \alpha}(f,0) \right)
  \geq
  \ell\left( \TPS^{(d)}_{\geq \alpha}(f,0) \right)
  -
  h\left( \TPS^{(d)}_{\geq \alpha}(f,0) \right).
\end{equation}
As a result,
\begin{equation} \label{eq:mono}
  \ell\left(\TPS^{(d+1)}_{\geq \alpha}(f,0)\right)
  \geq
  (\alpha-1) \ell\left(\TPS^{(d)}_{\geq \alpha}(f,0)\right).
\end{equation}
with a strict inequality if $\alpha < \deg \TPS^{(d+1)}(f,0)$.
\end{thm}

\begin{cor} \label{cor:upper_half}
The virtual vertices of $\tTPS(f,0)$ (see \cref{block:vertices})
lie in the closed upper halfplane.
\end{cor}

\begin{proof}
Since $f$ is singular, we have $m_q^{(d+1)} \geq 2n_q^{(d+1)}$ for
all $d$ and $q$. Therefore, \cref{eq:hlh} gives
\[
  h\left( \TPS^{(n+1)}_{\geq \alpha}(f,0) \right)
  \geq
  h\left( \TPS^{(n)}_{\geq \alpha}(f,0) \right)
  \geq
  \cdots
  \geq
  h\left( \TPS^{(0)}_{\geq \alpha}(f,0) \right)
  =
  1
\]
As a result, we have a nonnegative alternating sum
\[
  h\left( \tTPS_{\geq \alpha}(f,0) \right)
  =
  \sum_{d=0}^n
  (-1)^{n-d} h\left( \TPS^{(d+1)}_{\geq \alpha}(f,0) \right)
  \geq 0. \qedhere
\]
\end{proof}

\begin{cor} \label{cor:lc_degs}
Let $r$ be the rank of the Hessian matrix of $f$ at $0$
and assume that either
\[
  d+1 > r.
\]
or that $d$ is even (see \cref{rem:tps_props}\cref{it:AJ0}).
Then, the following inequalities hold for the
alternating Jacobian polygon:
\begin{enumerate}

\item \label{it:ttps_ineq}
For $\alpha \in \Q$, we have
\[
  \ell\left(\tTPS_{\geq \alpha}^{(d+1)}(f,0)\right)
  \geq
  0
\]
with equality if and only if $\alpha > \deg \TPS^{(d+1)}(f,0)$.

\item \label{it:lc}
The leading coefficient of the alternating Jacobian polygon is positive
\begin{equation} \label{eq:lcAJ}
  \lc\tTPS^{(d+1)}(f,0) > 0.
\end{equation}

\item \label{it:degs}
The degree of the Jacobian polygon increases with dimension
\begin{equation}
  \deg\TPS^{(d+1)}(f,0)
  \geq \deg\TPS^{(d)}(f,0), \qquad
  d = 1,\ldots,n.
\end{equation}

\item \label{it:both_degs}
The Jacobian and alternating Jacobian polygons have the same degree
\begin{equation} \label{eq:degJAJ}
  \deg \tTPS^{(d+1)}(f,0) = \deg \TPS^{(d+1)}(f,0).
\end{equation}
\end{enumerate}
\end{cor}
\begin{proof}
To prove \cref{it:ttps_ineq}, consider first the case
$\alpha \leq 2$. Since $\supp\TPS^{(d+1)}(f,0) \subset [2,\infty)$
by \cref{rem:tps_props}\cref{it:alpha_mult}, we have
$\TPS_{\geq \alpha}^{(d+1)}(f,0) = \TPS^{(d+1)}(f,0)$, and so
by \cref{rem:tps_props}\cref{it:tps_length} and \cref{it:tps_telescope}
\[
  \ell\left(\TPS^{(d+1)}_{\geq \alpha}(f,0)\right)
  =
  \mu^{(d+1)} + (-1)^d > 0.
\]
Consider next the case $2 < \alpha \leq \deg \TPS^{(d+1)}(f,0)$.
Then \cref{eq:mono} gives
\begin{equation} \label{eq:terms}
  \ell\left(\TPS^{(d+1)}_{\geq \alpha}(f,0)\right)
  \geq
  \ell\left(\TPS^{(d)}_{\geq \alpha}(f,0)\right)
  \geq \cdots \geq
  \ell\left(\TPS^{(0)}_{\geq \alpha}(f,0)\right)
  = 0
\end{equation}
with the first inequality strict.
Indeed, we have $\alpha - 1 > 1$ since $\alpha > 2$,
and $\ell\left(\TPS^{(d+1)}_{\geq \alpha}(f,0)\right) > 0$ since
$\alpha \leq \deg \TPS^{(d+1)}(f,0)$.
As a result, the alternating sum
\[
  \ell\left(\tTPS^{(d+1)}_{\geq \alpha}(f,0)\right)
  =
  \sum_{i=0}^{d+1} (-1)^i 
  \left(\TPS^{(d+1-i)}_{\geq \alpha}(f,0)\right)
\]
is strictly positive.
Next, consider the case $\alpha > \deg \TPS^{(d+1)}$.
Then all the inequalities in \cref{eq:terms} are equalities, and all
the terms are zero, proving 
\cref{it:ttps_ineq}, which immediately gives \cref{it:lc}.

Now, \cref{it:degs} follows immediately from \cref{eq:mono}.
Along with the definition of the alternating Jacobian polygon,
this gives $\leq$ in \cref{eq:degJAJ}. But $\geq$ follows from
\cref{eq:lcAJ} and \cref{eq:J_AJ_AJ}
\end{proof}

\begin{cor}
The \L ojasiewicz exponent increases with dimension
\[
\pushQED{\qed}
  \Loj\left(f^{(d+1)},0\right)
  \geq
  \Loj\left(f^{(d)},0\right), \qquad
  d = 1,2,\ldots,n.
  \qedhere
\popQED
\]
\end{cor}

\begin{cor} \label{cor:loj_reduced}
The \L ojasiewicz exponent of $f$ at $0$ is given by the degree
of the alternating Jacobian polygon minus one
\[
  \Loj(f,0) = \deg \tTPS(f,0) - 1
\]
unless $n$ is odd and $f$ has a Morse point at $0$. \qed
\end{cor}

\begin{block}
Let $\Delta$ be the \emph{Cerf diagram} associated with
the functions $f$ and $g$, that is, the image of 
the polar curve under the map
\[
  \Phi = (g,f):(\C^{n+1},0) \to (\C^2,0).
\]
We will use coordinates $u,v$ in $\C^2$, so that this map is given by
$u = g$ and $v = f$.
For any $a,b,c\in\C$, set
\[
  L_{a,b,c} = \set{(u,v)\in\C^2}{au + bv = c}.
\]
Writing the polar curve as a union of branches
$P = P_1 \cup \ldots \cup P_l$, denote by
$\Delta_q$ the image of $P_q$ under $\Phi$. Then
\[
  (X,P_q)_{\C^{n+1},0} = (L_{0,1,0}, \Delta_q)_{\C^2,0},\qquad
  (H,P_q)_{\C^{n+1},0} = (L_{1,0,0}, \Delta_q)_{\C^2,0}.
\]
Furthermore, we have the Milnor fiber
\[
  F_{\epsilon,\eta}
  = \Phi^{-1}\left(L_{0,1,\eta} \right)
    \cap B^{2n+2}_\epsilon
\]
and the map $F_{\epsilon,\eta} \to L_{0,1,\eta}$ has critical set
$\Sigma = P \cap F_{\epsilon,\eta}$. For each $q=1,\ldots,l$, we can write
\[
  \Delta_q \cap L_{0,1,\eta} = \{d_{q,1},\ldots,d_{q,m_q}\},\qquad
  P_q \cap F_{\epsilon,\eta} = \{c_{q,1},\ldots,c_{q,m_q}\}
\]
so that $\Phi(c_{q,j}) = d_{q,j}$ (recall $m_q = (X,P_q)_0$). We can lift these points to
continuous paths $c_{q,j}:[0,1] \to \C^2$ and $d_{q,j} \to \C^{n+1}$,
such that that for every $t \in [0,1]$
\[
\begin{split}
  \Delta_q \cap L_{-t\eta,1,(1-t)\eta}
  &= \{d_{q,1}(t),\ldots,d_{q,m_q}(t)\}, \\
  P_q \cap \Phi^{-1}(L_{-t\eta,1,(1-t)\eta}) \cap B_\epsilon^{2n+2}
  &= \{c_{q,1}(t),\ldots,c_{q,m_q}(t)\}
\end{split}
\]
and $c_{q,j}(0) = c_{q,j}$ and  $d_{q,j}(0) = d_{q,j}$.
\end{block}

\begin{figure}[ht]
\begin{center}
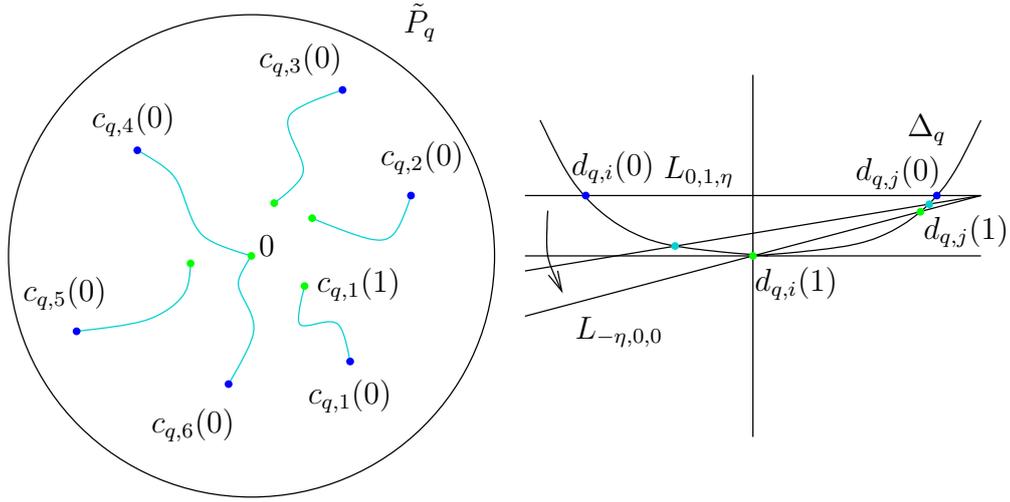
\caption{In this example we have $m_q = 6$ and $n_q = 2$.}
\label{fig:tilt}
\end{center}
\end{figure}

\begin{definition}
\begin{enumerate}

\item
For $q = 1,\ldots,l$, set
\[
  K_q^{(d+1)}
  =
  \set{ d^{(d+1)}_{q,l} \in \sigma^{(d+1)} }{ c^{(d+1)}_{q,l}(1) = 0 }
\]

\item
Denote by $(\hat C_\cdot^{(d+1)}, \hat \partial_\cdot)$
the subcomplex of $(C_\cdot, \partial_\cdot)$ generated
by $\Sigma^{(i)}$ in degree $i \leq d$, and by $K^{(d+1)}$
in degree $d$.

\item
The complex $(\hat C_\cdot, \hat \partial_\cdot)$ is a filtered
complex by setting
\[
  F_\alpha \hat C_d = F_\alpha C_d \cap \hat C_d.
\]

\end{enumerate}
\end{definition}

\begin{lemma}
For every $q$, we have
\[
  \left| K^{(d+1)}_q \right|
  =
  n_q^{(d+1)}.
\]
\end{lemma}

\begin{proof}
Recall $n_q^{(d+1)} = (H^{(d)}, P^{(d+1)}_q)_{H^{(d+1)},0}$
by \cref{def:m_n_alpha}.
For $t \neq 0$, the lines $L_{-t\eta,1,(1-t)\eta}$ are not parallel to
a tangent of $\Delta$, i.e. the $u$-axis. As a result,
if $0 \mleq t < 1$, then, in a neighborhood of $0\in\C^2$, there are precisely
\[
  \left(L_{-\eta, 1, 0}, \Delta_q^{(d+1)}\right)_{\C^2,0}
  = \mult(\Delta_q^{(d+1)},0) = n_q^{(d+1)}
\]
points in the
intersection $L_{-\eta, 1, 0} \cap \Delta_q^{(d+1)})_{\C^2,0}$, which
correspond precisely the points $c_{q,j}(t)$ which then converge to zero
when $t \to 1$.
\end{proof}

\begin{lemma} \label{lem:acyclic}
The complex $(\hat C_\cdot^{(d+1)},\hat \partial_\cdot)$ has the homology
of a ball.
\end{lemma}

\begin{proof}
Choose $\kappa > 0$ small, and $t \in [0,1]$ near $1$,
such that $|v(d_{q,j}(t))| < \kappa$ if and only if
$d_{q,j} \in K^{(d+1)}$. Then the set
\[
  \left(\Phi^{(d+1)}\right)^{-1}
  \left(L_{-t\eta,1,(1-t)\eta} \cap \{|v|<\kappa\}\right)
  \cap B^{2n+2}_\epsilon
\]
is homeomorphic to a Milnor fiber for $g^{(d+1)}$, which is a ball.
The complex $\hat C^{(d+1)}_\cdot$ is then the Morse--Smale complex
for a Morse function and gradient-like vector field constructed
for this set, similarly as in \cref{block:morse_smale}.
\end{proof}

\begin{lemma} \label{lem:nk_ineq}
Assume that $1 \leq d \leq n$. Then
\begin{enumerate}

\item \label{it:nk_im}
The images $\im \partial_d^{(d+1)}$ and $\im \hat \partial_d^{(d+1)}$
coincide.

\item \label{it:nk_ineq}
For $\alpha \in \Q$, we have
\begin{equation} \label{eq:nk}
  \rk \ker
  \left(
    \frac{\hat C^{(d+1)}_{d-1}}{F_\alpha \hat C^{(d+1)}_{d-1}}
    \overset{\partial}{\to}
    \frac{\hat C^{(d+1)}_{d-2}}{F_\alpha \hat C^{(d+1)}_{d-2}}
  \right)
  \geq
  \sum_{\alpha_q^{(d+1)} \geq \alpha}
    m_q^{(d)} - n_q^{(d)}.
\end{equation}
\end{enumerate}
\end{lemma}
\begin{proof}
First, we prove \cref{it:nk_im}.
If $d=1$, then the homology of either complex in degree $0$ is free of
rank one, and we have
\[
  \im \partial_1^{(2)} = \im \hat \partial_1^{(2)} = 
  \set{\sum_{p \in \Sigma^{(1)}} a_p p \in C_0}
      {\sum_{p \in \Sigma^{(1)}} a_p = 0}
\]
If $d > 1$, then the complexes
$(C^{(d+1)}_\cdot, \partial^{(d+1)}_\cdot)$ and
$(\hat C^{(d+1)}_\cdot, \hat \partial^{(d+1)}_\cdot)$
are exact in degree $d-1$, and so both images coincide
with the kernel of $\partial_{d-1}^{(d+1)} = \hat \partial_{d-1}^{(d+1)}$
\[
    \im \partial_d^{(d+1)}
  = \ker \partial_{d-1}^{(d+1)}
  = \ker \hat \partial_{d-1}^{(d+1)}
  = \im \hat \partial_d^{(d+1)}.
\]

Next, we prove \cref{it:nk_ineq} using \cref{it:nk_im}.
For every element,
$x \in \Sigma^{(d+1)} \setminus K^{(d+1)}$, there exists a linear
combination $y_x$ of elements of $K^{(d+1)}$ such that
$\hat \partial(y_x) = \partial(x)$. As a result, the elements
$x - y_x$ for $x \in P_q^{(d)}$ satisfying
$\alpha_q^{(d)} \geq \alpha$ form a linearly independent
set of vectors modulo $F_\alpha$ in the kernel of the map $\partial$
on the left in \cref{eq:nk}.
For each $q$ such that $\alpha_q^{(d)} \geq \alpha$, there are
precisely $m_q^{(d)} - n_q^{(d)}$ of them.
\end{proof}

\begin{lemma} \label{lem:C_iso}
For $d=1,\ldots,n$, the morphism
\begin{equation} \label{eq:C_iso}
  \hat C_d^{(d+1)} \to \frac{\hat C^{(d+1)}_{d-1}}{\hat C^{(d)}_{d-1}}
\end{equation}
induced by $\hat \partial^{(d+1)}_d$ is an isomorphism.
\end{lemma}
\begin{proof}
The inclusion $\hat C^{(d)}_\cdot \subset \hat C^{(d+1)}_\cdot$ corresponds
to the inclusion of a ball in another ball. In particular, it induces
an isomorphism on homology, and so the corresponding
quotient complex is acyclic. In particular, we have an exact
sequence
\[
  \to \frac{ \hat C^{(d+1)}_{d+1} }{ \hat C^{(d)}_{d+1} }
  \to \frac{ \hat C^{(d+1)}_d     }{ \hat C^{(d)}_d     }
  \to \frac{ \hat C^{(d+1)}_{d-1} }{ \hat C^{(d)}_{d-1} }
  \to \frac{ \hat C^{(d+1)}_{d-2} }{ \hat C^{(d)}_{d-2} }
  \to
\]
The term to the left is $0$, since $C^{(d+1)}_{d+1} = 0$, and
the term to the right is $0$, since 
$\hat C^{(d+1)}_{d-2} = C_{d-2} = \hat C^{(d)}_{d-2}$.
As a result, the middle morphism is an isomorphism, which shows that
\cref{eq:C_iso} is an isomorphism, since $\hat C^{(d)}_d = 0$.
\end{proof}

\begin{proof}[Proof of \cref{thm:mono}]
For a fixed $\alpha$,
the height on the left hand side of \cref{eq:hlh} consists of terms
of the type $n_q^{(d+1)}$ with $\alpha^{(d+1)}_q \leq \alpha$.
As a result,
\[
  \ell\left(\TPS^{(d+1)}_{\geq 0}(f,0)\right)
  \geq
  \alpha h\left(\TPS^{(d+1)}_{\geq 0}(f,0)\right).
\]
Using the same argument for the height on the right hand side,
\cref{eq:mono} therefore results from multiplying \cref{eq:hlh} by $\alpha$.
Thus, we focus on \cref{eq:hlh}.

Assuming first that $d = 1$, set $e = \mult(X,0)$.
If $\alpha > e$, then, by \cref{rem:tps_props}\cref{it:n0},
the right hand side of \cref{eq:hlh} vanishes. Since the right
hand side is nonnegative, the inequality holds.
If, however, $\alpha \leq e$, then \cref{eq:hlh} reads as
\[
  \mu^{(1)} \geq \mu^{(1)} + \mu^{(0)} - \mu^{(0)}
\]
by \cref{rem:tps_props}\cref{it:alpha_mult} and \cref{it:tps_length},
which clearly holds.

For the rest of this proof, we assume that $d > 1$.
By \cref{lem:C_iso}, for any $\alpha \in \Q$, the induced map
\[
  \frac{\hat C^{(d+1)}_d}{F_\alpha \hat C^{(d+1)}_d}
  \to
  \frac{\hat C^{(d+1)}_{d-1}}
       {\hat C^{(d)}_{d-1} + F_\alpha \hat C^{(d+1)}_{d-1}}
\]
is surjective. As a result,
\begin{equation} \label{eq:rk_ineq}
  \rk \frac{\hat C^{(d+1)}_d}{F_\alpha \hat C^{(d+1)}_d}
  \geq
  \rk \frac{\hat C^{(d+1)}_{d-1}}
       {\hat C^{(d)}_{d-1} + F_\alpha \hat C^{(d+1)}_{d-1}}
\end{equation}
On one hand, the left side of \cref{eq:rk_ineq} is freely generated by
$p \in K_q^{(d+1)}$, and so
\begin{equation} \label{eq:one_hand}
  \rk \frac{\hat C^{(d+1)}_d}
           {F^{\vphantom{d}}_{\vphantom{d}\alpha} \hat C^{(d+1)}_d}
  =
  \sum_{\alpha^{(d+1)}_q \geq \alpha} n^{(d+1)}_q
  =
  h\left(\TPS^{(d+1)}_{\geq \alpha}(f,0)\right).
\end{equation}
On the other hand, the right side of \cref{eq:rk_ineq} is freely generated
by $p \in \Sigma_q^{(d)} \setminus K_q^{(d)}$, and so
\begin{equation} \label{eq:other_hand}
  \rk \frac{\hat C^{(d+1)}_{d-1}}
       {\hat C^{(d)}_{d-1} + F_\alpha \hat C^{(d+1)}_{d-1}}
  =
  \sum_{\alpha_q^{(d)} \geq \alpha} m_q^{(d)} - n_q^{(d)}
  =
  \ell\left(\TPS^{(d)}_{\geq \alpha}(f,0)\right)
  -
  h\left(\TPS^{(d)}_{\geq \alpha}(f,0)\right).\qedhere
\end{equation}
\end{proof}

\begin{block} \label{block:key}
In the above proof, we prove the key inequality
\begin{equation} \label{eq:key}
  \sum_{\alpha^{(d+1)}_q \geq \alpha} n^{(d+1)}_q
  \geq
  \sum_{\alpha_q^{(d)} \geq \alpha} m_q^{(d)} - n_q^{(d)}.
\end{equation}
We provide here a purely visual argument for this equality,
sketched out in \cref{fig:traj_scheme}.
What lies below a black horizontal line is a Milnor fiber $F^{(d+1)}$.
The blue and red points
are generators of the Morse--Smale complex $(C_\cdot, \partial_\cdot)$,
the red points correspond to the elements of $K_q^{(d+1)}$ or
$K_r^{(d)}$, and
the blue points to elements of
$\Sigma^{(d+1)}_q \setminus K_q^{(d+1)}$
or
$\Sigma^{(d)}_r \setminus K_r^{(d)}$.
Thus, the blue points correspond to handles freely generating the homology
of $F^{(d+1)}$ or $F^{(d)}$, whereas the red points are handles which
cancel out blue handles below. Thus, what lies below the yellow line
is a ball. 

\begin{figure}[ht]
\begin{center}
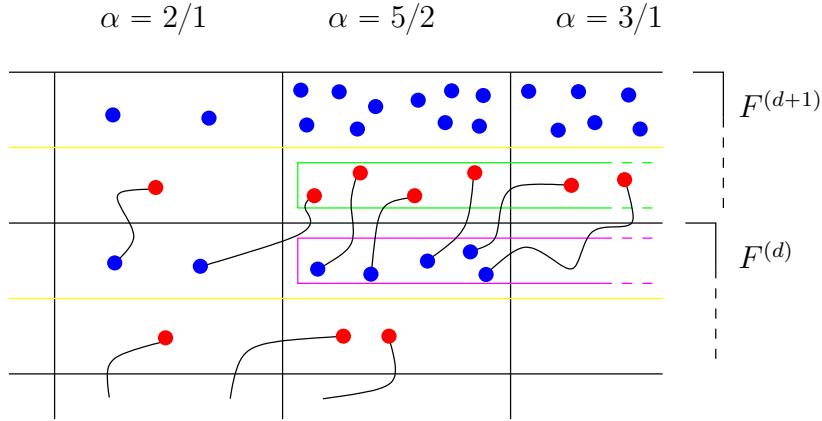
\caption{A visualization of the Morse--Smale complex.}
\label{fig:traj_scheme}
\end{center}
\end{figure}

The further to the right, the higher the vanishing rate of the
generators. The blue points in $F^{(d)}$ must be cancelled by
the red points in $F^{(d+1)}$. This means that there must be
enough trajectories passing up from the blue points to the red points.
Since trajectories going up cannot go to the left by
\cref{lem:trajs}, the number of red points in the green box must
be greater or equal to the number of blue points in the magenta box,
and this is precisely \cref{eq:key}.

\end{block}

\section{Newton nondegenerate functions} \label{s:nondeg}

\begin{block}
In this section, we introduce notation related to the Newton polyhedron
associated with a power series.
We will use freely the concept of being \emph{Newton nondegenerate},
introduced in \cite{kouchnirenko}.
We use coordinates $z_0,z_1,\ldots,z_n$ in $\C^{n+1}$ which are
\emph{not} generic in the sense of the previous sections.

The proof of \cref{thm:tps_nondeg}\cref{it:tps_nondeg_n}
is similar to that of Varchenko's formula for the monodromy zeta
function \cite{Var_zeta}, with
\cref{thm:tps} playing the role of A'Campos formula.
The more general \cref{thm:tps_nondeg}\cref{it:tps_nondeg_d}
is similar to Oka's formula for principal zeta-functions \cite{oka_principal}.
By taking lengths on both sides of \cref{eq:AJ_vol} and using
\cref{rem:tps_props}\cref{it:tps_telescope}, we recover
Kouchnirenko's formula \cite{kouchnirenko}.

We use the notion of \emph{sedentarity}, a concept from tropical
geometry, used in e.g. \cite{mik_zhar_wave,IKMZ}.
\end{block}

\begin{block}
Let $f \in \C\{z_0,z_1,\ldots,z_n\}$ be a convergent power series in
$n+1$ variable:
\[
  f(z) = \sum_{i \in \Z_{\geq 0}^{n+1}} a_i z^i.
\]
We will assume that $f$ is not identically zero,
but vanishes at zero, with at most an isolated critical point.
Furthermore, we assume that $f$ is \emph{Newton nondegenerate}, as
follows. The \emph{support} of $f$ is the set
\[
  \supp(f) = \set{i \in \Z^{n+1}}{a_i \neq 0}
\]
and its \emph{Newton polyhedron} is the convex closure of the union
of positive orthants translated by the support:
\[
  \Gamma_+(f)
  =
  \convx
  \left(
  \bigcup_{i \in \supp(f)} i + \R_{\geq 0}^{n+1}.
  \right)
\]
Using the interval
\[
  \bR_{\geq 0} = [0,+\infty]
\]
any $v = (v_0,\ldots,v_n) \in \bR_{\geq 0}^{n+1}$
will be identified with the function
\[
  v: \R^{n+1}_{\geq 0} \to \bR_{\geq 0},\qquad
  (u_0,u_1,\ldots,u_n) \mapsto \sum_{i=0}^n v_i u_i.
\]
The \emph{refined sedentarity} of such a vector, or function, is
\[
  \sed(v) = \set{i \in \{0,\ldots,n\}}{ v_i = +\infty}.
\]
The \emph{sedentarity} of $v$ is $|\sed(v)|$.
and we define
\[
  \wedge_f v = \min_{\Gamma_+(f)} v,\qquad
  K_f(v)
  =
\begin{cases}
  \set{u\in \Gamma_+(f)}{v(u) = \wedge_f(v)} & \wedge_f v < +\infty,\\
  \emptyset                                  & \wedge_f v = +\infty.
\end{cases}
\]
A \emph{face} of $\Gamma_+$ is a nonempty set of the form $K_f(v)$.
We denote by $K^\circ_f(v)$ its relative inerior, i.e. the topological
interior of the face seen as a subset of the affine
subspace it generates in $\R^{n+1}$.
Such a face is compact if and only if all the coordinates of
the vector $v$ are positive, i.e. positive real numbers or $+\infty$.
The \emph{Newton diagram} of $f$ is the union of all compact faces
of $\Gamma_+(f)$
\[
  \Gamma(f) = \bigcup_{v \in \bR^{n+1}_{> 0}} K_f(v),
\]
and we define $\Gamma_-(f)$ as the union of segments joining the origin
$0 \in \R^{n+1}$ and some point in $\Gamma(f)$. 
If $K = K_f(v)$ is a compact face of the Newton diagram, then there exists a
unique $I_K \subset \{0,\ldots,n\}$ such that
\[
  K^\circ \subset \bR_{>0}^{I_K}
  = \set{(u_0,\ldots,u_n) \in \bR_{\geq 0}^{n+1}}
        {u_i = 0 \Leftrightarrow i \in I_K}.
\]
If $\dim(K) = |I_K|-1$, then $K$ is a \emph{coordinate facet} of
the diagram $\Gamma(f)$. A coordinate facet of dimension $n$ is
simply a facet. If $K$ is a coordinate facet, then there exists a unique
vector $v_K$ called
\emph{the (primitive integral) normal vector} of $K$, satisfying
\begin{equation} \label{eq:primitive}
  v_K \in \bZ^{n+1}_{\geq 0}, \qquad
  \gcd\left(\set{v_{K,i}}{i \not\in \sed(v)}\right) = 1,
\end{equation}
such that $K = K(v_K)$ and $\sed(v_K) = \{0,1,\ldots,n\} \setminus I_K$. 
\end{block}

\begin{definition}
\begin{blist}
\item
Denote by $\CF(\Gamma)$ the set of coordinate facets of the Newton diagram
$\Gamma$, and by $\NCF(\Gamma)$ the set of their primitive integral normal
vectors. If $F$ is a coordinate facet corresponding to the primitive integral
normal vector $v$, then we set

\[
  m_F = m_v = \wedge_f v.
\]
\item
We will denote by $g \in \C\{z_0,z_1,\ldots,z_n\}$ a
\emph{generic linear function}, that is,
\[
  g(z_0,z_1,\ldots,z_n) = b_0 z_0 + b_1 z_1 + \ldots + b_n z_n
\]
for some generic cofficients $b_i \in \C$. We will, furthermore, always
assume that none of the $b_i$ vanish. As a result, the Newton diagram
of $g$ is the standard $n$-simplex in $\R^{n+1}$, i.e. the convex hull of the
natural basis. We set
\[
  n_F = n_v = \wedge_g v.
\]
\end{blist}
\end{definition}

\begin{definition}
Denote by $V_s$ the \emph{generalized mixed $s$-volume} in $\R^n$.
See e.g. \cite{bern_mixed_vol,oka_principal} for further details.
If $K_1,\ldots,K_s$ are convex bodies in $\R^n$, and
there exist translates $K'_i$ of $K_i$ each contained in the same
rational $s$-plane $L \subset \R^n$, then
\[
  V_s(K_1,\ldots,K_s) \in \R
\]
is the mixed volume of the convex bodies $K'_i$ with respect to
the Lebesuge measure $\Vol_s$ on $L$, normalized so that the parallelpiped
spanned by an integral basis of $L\cap\Z^n$ has volume $1$.
If $k_1,\ldots,k_c$ are nonnegative integers that sum up to $s$, then
$V_s(K^{k_1}_1,\ldots,K^{k_c}_c)$ means that each $K_i$ is
repeated $k_i$ times.
As in \cite{oka_principal}, but not as in \cite{bern_mixed_vol},
the mixed volume is normalized
in such a way that if $K$ is a convex set in an $s$-dimensional affine
subspace of $\R^{n+1}$, then
\begin{equation} \label{eq:vol_mv}
  V_s(K^s) = \Vol_s(K).
\end{equation}
\end{definition}

\begin{definition} \label{def:W}
Let $v \in \bZ^{n+1}_{> 0}$ be primitive, and $d \in \{0,1,\ldots,n\}$.
With $s = n - |\sed(v)|$ and $c = n-d$, define
\begin{equation} \label{eq:Wd}
  W^{(d+1)}(v)
  =
  \sum_{ k_0, k_1, \ldots, k_c }
  s!
  V_s \left(
    K_f(v)^{k_0},
    K_g(v)^{k_1},
    \ldots,
    K_g(v)^{k_c}
  \right)
\end{equation}
where the sum runs through $(c+1)$-tuples of integers
\[
  (k_0,k_1,\ldots,k_c) \in \Z_{\geq 0} \times \Z_{>0}^c,
  \qquad
  \sum_{j=0}^c k_j = s.
\]
\end{definition}

\begin{rem} \label{rem:W}
\begin{blist}

\item \label{it:W_svol}
In the case $n=d$ we have $c = 0$, and so by \cref{eq:vol_mv},
if $v \in \bZ_{>0}^{n+1}$ with $s = n-|\sed(v)|$, then
\begin{equation} \label{eq:W_vol}
  W^{(n+1)}(v) = s! \Vol_s(K_f(v)).
\end{equation}
In particular, $W^{(n+1)}(v) = 0$ unless $v \in \NCF(\Gamma(f))$.

\item \label{it:W_dle}
The polytopes $K_f(v)$ and $K_g(v)$ are contained in an $s$-dimensional
affine
subspace of $\R^{n+1}$. Using the normalized volume on this subspace,
along with the notation used in \cite{khovanskii_genus},
we have
\[
  W^{(d+1)}(v)
  =
  (-1)^{s-c}
  \left(
  \frac{1}{1+K_f(v)}
  \right)
  \left(
  \frac{K_g(v)}{1+K_g(v)}
  \right)^c_s
\]
What this means is the following:
Treat $K_f(v)$ and $K_g(v)$ as variables, and expand the rational function
at the origin. For every homogeneous monomial $K_f(v)^a K_g(v)^b$,
with $a+b = s$, evaluate the mixed volume
 $V_s(K_f(v)^a, K_g(v)^b)$.
The right hand side above is the sum of all such terms.
Counting up the number of possible $k_1,k_2, \ldots, k_c$, we find, if
$d < n$,
\begin{equation}
  W^{(d+1)}(v)
  =
  \sum_{k=c}^s
  \left(
  \begin{matrix}
   k-1 \\
   c-1 
  \end{matrix}
  \right)
  s!V_s \left(
    K_f(v)^{s-k},
    K_g(v)^k
  \right)
\end{equation}

\item \label{it:W_finite}
If $\dim(K_f(v) + K_g(v)) < s$, then every term in \cref{eq:Wd} vanishes.
As a result, we have $W^{(d+1)}(v) = 0$ unless
$v \in \NCF(\Gamma(f g))$. In particular, $W^{(d+1)}(v) \neq 0$ for
all but finitely many primitive vectors $v \in \bZ_{>0}^{n+1}$.

\end{blist}
\end{rem}

\begin{thm} \label{thm:tps_nondeg}
\begin{blist}

\item \label{it:tps_nondeg_n}
We have
\begin{equation} \label{eq:AJ_vol}
  \tTPS(f,0)
  =
  \sum_{F \in \CF}
  (-1)^{n-s}
  s!
  \Vol_s F
  \newpol{ m_F }{ n_F }.
\end{equation}
where for any $F \in \CF$, we set $s = \dim(F)$.

\item \label{it:tps_nondeg_d}
For any $d=0,1,\ldots,n$,
we have $W^{(d+1)}(v) = 0$ for all but finitely many primitive
integral vectors $v \in \bZ_{>0}^{n+1}$, and
\begin{equation} \label{eq:AJ_mv}
  \tTPS^{(d+1)}(f,0)
  =
  \sum_{\substack{v \in \bZ_{>0}^{n+1} \\ \mathrm{primitive} }}
  (-1)^{n-s}
  W^{(d+1)}(v)
  \,
  \newpol{ m_v }{ n_v }.
\end{equation}

\end{blist}
\end{thm}

\begin{cor} \label{cor:tps_nondeg}
The Jacobian polygons associated with $f$ are given by
\begin{equation} \label{eq:AJ_mv_TPS}
  \TPS^{(d+1)}(f,0)
  =
  \sum_{\substack{v \in \bZ_{>0}^{n+1} \\ \mathrm{primitive} }}
  (-1)^{n-s}
  \left(
  W^{(d+1)}(v)
  +
  W^{(d)}(v)
  \right)
  \,
  \newpol{ m_v }{ n_v }.
\end{equation}
\end{cor}
\begin{proof}
This follows from the above theorem, and
\cref{rem:tps_props}\cref{it:two_summands}.
\end{proof}

\begin{figure}[ht]
\begin{center}
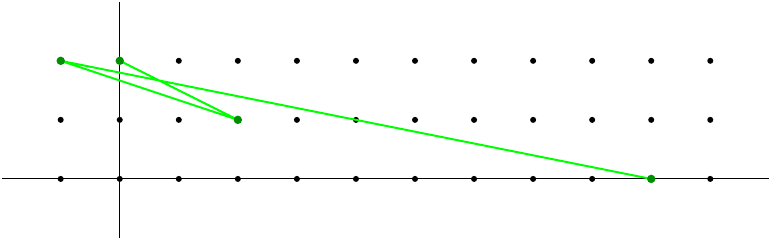
\caption{The alternating Jacobian polyhedron of the $E_8$ surface singularity.}
\label{fig:235}
\end{center}
\end{figure}

\begin{example}
\begin{blist}
\item \label{ex:235}
The $E_8$ singularity $f(x,y,z) = x^2 + y^3 + z^5$ is Newton nondegenerate.
By \cref{thm:tps_nondeg}, its alternating Jacobian polyhedron is
\[
  \tTPS(f,0)
  =
  2 \newpol{5}{1}
  - \newpol{3}{1}
  + \newpol{2}{1}.
\]
In \cref{fig:235}, we see its virtual vertices and edges
(see \cref{block:vertices}). As predicted by \cref{cor:upper_half},
the vertices lie in the upper halfplane. One vertex, however,
strays over to the left halfplane.

\item
In \cite[Example (3.7) and (3.8)]{fukui_newton_diag},
Fukui considers the examples
\[
  f_\delta(x_1,x_2,x_3,x_4)
  =
  x_1^2 x_2^2 x_3^2 x_4^2
  + x_1 x_3^8
  + x_2 x_4^8
  + x_1^8 x_4
  + x_2^8 x_3
  + \delta x_2^4 x_3^4
\]
where $\delta = 0,1$. One checks that every face of this diagram is
a simplex, and so the functions are Newton nondegenerate.
Using \cref{thm:tps_nondeg}, we find

\begin{minipage}{0.5\textwidth}
\[
\begin{split}
  \tTPS^{(4)}(f_0,0)
  &= 8 \newpol{455}{47} \\
  \tTPS^{(3)}(f_0,0)
  &= \newpol{8}{1} + 40\newpol{9}{1} + 2 \newpol{28}{3} \\
  \tTPS^{(2)}(f_0,0)
  &= 2 \newpol{8}{1} + 4 \newpol{9}{1} \\
  \tTPS^{(1)}(f_0,0)
  &= \newpol{1}{8}
\end{split}
\]
\end{minipage}
\begin{minipage}{0.5\textwidth}
\[
\begin{split}
  \tTPS^{(4)}(f_1,0)
  &= 57 \newpol{28}{3} + 4 \newpol{455}{47} \\
  \tTPS^{(3)}(f_1,0)
  &= 4 \newpol{8}{1} + 29 \newpol{9}{1} + 4 \newpol{28}{3} \\
  \tTPS^{(2)}(f_1,0)
  &= 4 \newpol{8}{1} + 2 \newpol{9}{1} \\
  \tTPS^{(1)}(f_0,0)
  &= \newpol{1}{8}
\end{split}
\]
\end{minipage}

\vspace{.1cm}

Fukui showed that in either case, the inequality
\[
  \Loj(f_\delta,0) \leq 8 + \frac{32}{47}
\]
holds. By our computations and \cref{cor:loj_reduced}, in fact, equality
holds.
\end{blist}
\end{example}

\begin{proof}[Proof of \cref{thm:tps_nondeg}]
It follows from \cref{rem:W}\cref{it:W_svol} that
\cref{it:tps_nondeg_d} implies \cref{it:tps_nondeg_n}.
Furthermore, if $v \in \bZ^{n+1}$ is primitive, then
we have $W^{(d+1)}(v) = 0$ unless $v \in \NCF(fg)$
by \cref{rem:W}\cref{it:W_finite}.
We are therefore left with proving \cref{eq:AJ_mv}.

Choose a regular subdivision $\triangle$ of the cone $\R_{\geq 0}^{n+1}$
which refines the natural subdivisions $\triangle_f, \triangle_g$
dual to $\Gamma(f)$ and $\Gamma(g)$.
Note that as a result, any normal vector to a facet of $\Gamma(fg)$
therefore generates a ray in $\triangle$.
This way, $\triangle$ is a fan which
induces a toric variety $Y$ and a toric map
$\pi_\triangle:Y\to \C^{n+1}$. That $\triangle$ refines $\triangle_f$
implies that $\pi_\triangle$ is an embedded resolution of $f$.
The requirement that $\triangle$ refines $\triangle_g$ is equivalent
to requiring that $\pi_\triangle$ factors through the blow-up of
$\C^{n+1}$ at $0$. 
We use the notation introduced in \cref{s:res} for $\pi = \pi_\triangle$.

There is a natural map $v:I \to \Z^{n+1}_{>0}$
such that if $D_i \subset Y$ is an
irreducible exceptional component, then $v(i) \in \Z^{n+1}_{>0}$ is prime,
and $D_i$ equals the orbit closure $\overline O_{v(i)}$.
The subset $D_i^\circ$ is the union of those orbits $O_\sigma$ corresponding
to cones $\sigma$ generated by $v(i)$ and some subset $S$ of the natural
basis $e_0, e_1, \ldots, e_n$ of $\Z^{n+1}$.
For such a $\sigma$, denote by $v_\sigma$ the primitive vector
in $\bZ_{>0}^{n+1}$ obtained by replacing the $j$th coordinate
of $v$ by $\infty$ for $j \in S$, and dividing by the greatest common
divisor of the remaining coordinates, e.g.
\[
  v_i = (2,3,6), \quad S = \{0\}
  \qquad
  \rightsquigarrow
  \qquad
  v_\sigma = (\infty, 1, 2).
\]
In this case, set $O_v = O_\sigma$.
Any vector in $\NCF(fg)$ arises in this fashion from some
$\sigma \in \triangle$. Thus, applying \cref{thm:tps}, we find
\begin{equation} \label{eq:chi_d}
\begin{split}
  \tTPS^{(d+1)}(f,0)
  &= (-1)^d \sum_{i\in I}
      \chi\left( D_i^\circ \cap Y^{(d+1)} \setminus \tX \right) 
      \newpol{m_i}{n_i} \\
  &= (-1)^d \sum_{\substack{v \in \bZ_{>0}^{n+1} \\ \mathrm{primitive} }}
      \chi\left( O_v \cap Y^{(d+1)} \setminus \tX \right)
      \newpol{m_v}{n_v}
\end{split}
\end{equation}
For any such $v = v_\sigma$, we have
\begin{equation} \label{eq:chis}
  \chi\left( O_v \cap Y^{(d+1)} \setminus \tX \right)
  =
  \chi\left( O_v \cap Y^{(d+1)} \right)
  -
  \chi\left( O_v \cap Y^{(d+1)} \cap \tX \right).
\end{equation}
In \cite[\textsection 6]{oka_principal}, Oka proves that
\[
  \chi\left( O_v \cap Y^{(d+1)} \right)
  =
  \left(
    \frac{K_g(v)}{1+K_g(v)}
  \right)^c_s
\]
and
\[
  \chi\left( O_v \cap Y^{(d+1)} \cap \tX \right)
  =
  \left(
    \frac{K_f(v)}{1+K_f(v)}
  \right)
  \left(
    \frac{K_g(v)}{1+K_g(v)}
  \right)^c_s.
\]
Thus, we have
\[
\begin{split}
  \chi\left( O_v \cap Y^{(d+1)} \setminus \tX \right)
  &=
  \left(
    1 - \frac{K_f(v)}{1+K_f(v)}
  \right)
  \left(
    \frac{K_g(v)}{1+K_g(v)}
  \right)^c_s \\
  &=
  \left(
    \frac{1}{1+K_f(v)}
  \right)
  \left(
    \frac{K_g(v)}{1+K_g(v)}
  \right)^c_s \\
  &=
  (-1)^{s-c} W^{(d+1)}(v).
\end{split}
\]
Since $(-1)^d(-1)^{s-c} = (-1)^{n-s}$, \cref{eq:chi_d} gives \cref{eq:AJ_mv}.
\end{proof}

\section{The \L ojasiewicz exponent from the Newton diagram} \label{s:loj}

\begin{block}
In this section we give a formula for the \L ojasiewicz exponent
of a Newton nondegenerate function in terms of its Newton diagram,
\cref{thm:bko_statement}, which is a direct result of the previous
results of this paper.
As a corollary, we recover a result of Brzostowski \cite{Brzostowski}.
\end{block}

\begin{definition} \label{def:maxax}
Let $\Gamma \subset \R_{\geq 0}^{n+1}$ be a Newton diagram, i.e. the
union of compact faces of a Newton polyhedron.
Denote by $\Lambda$ the Newton diagram of a generic linear function
\[
  \Lambda = \Gamma\left( \sum_{k=0}^n b_k z_k \right), \qquad
  b_0, b_1, \ldots, b_n \in \C^*.
\]
\begin{blist}

\item
If $F \in \CF(\Gamma)$ is a coordinate facet, corresponding to a primitive
normal vector $v \in \NCF(\Gamma)$, then the
\emph{maximal axial intersection} of $F$, or $v$, is
\[
  \maxax(F) = \maxax(v) = \frac{\wedge_\Gamma v}{\wedge_\Lambda v}.
\]

\item
The \emph{maximal axial intersection} of $\Gamma$ is
\[
  \maxax(\Gamma)
  =
  \max\set{ \frac{\wedge_\Gamma v}{\wedge_\Lambda v} }
          { v \in \NCF(\Gamma) }.
\]

\item
For $\alpha \in \Q$, let
\[
  s_\alpha(\Gamma)
  =
  \bigcup \set{ K_f(v) } { v \in \bR_{>0}^{n+1},\; \maxax(v) \leq \alpha }
  \subset
  \Gamma.
\]

\end{blist}
\end{definition}

\begin{definition}[\cite{kouchnirenko}] \label{def:newton_number}
\begin{blist}
\item
If $S \subset \R^{n+1}$ is a measurable subset, define the
\emph{Newton number}
\[
  \nu(S) = \sum_{J \subset \{0,1,\ldots,n\}}
    (-1)^{n+1-|J|}
    |J|!\Vol_{|J|}\left(S \cap \R^J\right)
\]
\item
If $\Gamma \subset \R^{n+1}$ is a Newton diagram, denote by
$\Gamma_-$ the union of segments joining any point in $\Gamma$ with
the origin.
\end{blist}
\end{definition}

\begin{thm} \label{thm:bko_statement}
Let $f \in \C\{z_0, z_1, \ldots, z_n \}$ have a Newton nondegenerate
isolated singularity at $0$, with Newton diagram $\Gamma = \Gamma(f)$.
Assume that either $n$ is even or that $f$ does not have a Morse
point at $0$.
Then, there exists a coordinate facet $F \subset \Gamma$ such that
\begin{equation} \label{eq:bko_F}
  \Loj(f,0) = \maxax(F) - 1.
\end{equation}
For any $\alpha \in \Q$, we have
\begin{equation} \label{eq:bko_nu}
  \nu(\Gamma_-) \geq \nu(s_\alpha(\Gamma)_-)
\end{equation}
with an equality if and only if $\alpha - 1 \geq \Loj(f,0)$. In particular,
\begin{equation} \label{eq:bko_min}
  \Loj(f,0)
  =
  \min\set{ \alpha - 1 \in \Q }
          { \nu(s_\alpha(\Gamma(f))_-) = \nu(\Gamma_-) }.
\end{equation}
\end{thm}
\begin{proof}
The existence of $F$ follows immediately from
\cref{thm:tps_nondeg}\cref{it:tps_nondeg_n}
and \cref{cor:loj_reduced}.
The rest follows from \cref{cor:lc_degs}\cref{it:ttps_ineq}.
\end{proof}

\begin{cor}[\cite{Brzostowski}]
If $f,f' \in \C\{z_0,z_1,\ldots,z_n\}$ are Newton nondegenerate
and $\Gamma(f) = \Gamma(f')$, then
\[
\pushQED{\qed}
  \Loj(f,0) = \Loj(f',0).
  \qedhere
\popQED
\]
\end{cor}

\begin{rem}
If $f$ has a Morse point at $0$, then $\Loj(f,0) = 1$.
In the case when $n$ is odd, the function
\[
  f(z_0,\ldots,z_n) = z_0z_1 + z_2z_3 + \ldots + z_{n-1} z_n
\]
is Newton nondegenerate and has an isolated singularity, but
$\Gamma(f)$ has no coordinate facet.
\end{rem}

\begin{question}
\begin{blist}
\item
If $\alpha = \Loj(f,0)+1$, and $f' \in \C\{z_0,\ldots,z_n\}$ is Newton
nondegenerate such that $\Gamma(f') = s_\alpha(\Gamma(f))$, does
$f'$ then have an isolated singularity?
Note that by Kouchnirenko's criterion \cite{kouchnirenko},
this property only depends on $\Gamma(f')$.
\item
More generally, assume that $f'$ is Newton nondegenerate and that
$\Gamma(f') \subset \Gamma(f)$.
Does $f'$ then have an isolated singularity, if the two diagrams give rise
to the same Newton number, i.e. if
$\nu(\Gamma_-(f)) = \nu(\Gamma_-(f'))$?
\item
Conversely, if we assume that $f$ and $f'$ are Newton nondegenerate,
have isolated singularities, and that $\Gamma(f') \subset \Gamma(f)$, does it
then follow that $f$ and $f'$ have the same Milnor number?
By semicontinuity of the Milnor number, we know that in this case,
we have $\mu(f',0) \geq \mu(f,0)$.
Can it happen that
\[
  \nu(\Gamma_-(f')) > \nu(\Gamma_-(f)) \; ?
\]
\end{blist}
\end{question}

\section{Past, present and future} \label{s:ppf}

\begin{block}
In this section we start by recalling a conjecture by
Brzostowski, Krasiński and Oleksik \cite{bko_conj}
which greatly motivated this manuscript.
We then give a counterexample to this conjecture.
Finally, we give a similar conjecture which 
draws inspiration from the nonnegativity of Stanley's local $h$-vector
\cite{stanley_subdivisions} and a formula for the Newton number
\cite{selyanin_mono}.

In \cref{def:CNC}, we use rational coefficients in $K\new$, i.e.
we work implicitly in the group $K\new \otimes_\Z \Q$.
Given that $K\new$ is a free Abelian group, this should not cause
any confusion.
\end{block}

\begin{definition}[\cite{bko_conj}] \label{def:bko_exc}
Let $f \in \C\{x_0,\ldots,x_n\}$. A facet $F \subset \Gamma(f)$ of
its Newton diagram is \emph{exceptional} if there exist $i\neq j$
such that all but one vertex of $F$ lies on the $j$-th coordinate
hyperplane, and that this one vertex corresponds to a monomial
of the form $x_j x_i^k$. Denote by $E_f$ the set of
exceptional facets of $\Gamma(f)$.
\end{definition}

\begin{conj}[\cite{bko_conj}] \label{conj:bko}
Let $f \in \C\{x_0,\ldots,x_n\}$ be Newton nondegenerate, and assume that
the diagram $\Gamma(f)$ contains a nonexceptional facet
$F \in \CF \setminus E_f$, $\dim F = n$. Then
\[
  \Loj(f,0)
  =
  \max \set{\maxax(F) - 1 \in \Q}{F \in \CF \setminus E_f, \dim F = n}.
\]
\end{conj}

\begin{example} \label{example:counter}
Consider the function $f \in \C\{x,y,z,w\}$ with generic coefficients in
front of the monomials
\[
  x^2,\; y^2,\; xz,\; xw,\; yz,\; yw,\; z^3,\; w^3.
\]
In particular, we assume that $f$ is Newton nondegenerate.
A computation shows that the Hessian of $f$ at $0$ is nonzero,
and so $f$ has a Morse point at $0$. In particular, we have
\[
  \Loj(f,0) = 1.
\]
The Newton diagram $\Gamma(f)$ has two facets, say,
$F_1$ with normal vector $v_1 = (1,1,1,1)$, whose
vertices correspond to the monomials
\[
 x^2,\; y^2,\; xz,\; xw,\; yz,\; yw
\]
and $F_2$, with normal vector $(2,2,1,1)$, whose
vertices correspond to the monomials
\[
  xz,\; xw,\; yz,\; yw,\; z^3,\; w^3.
\]
Neither facet is exceptional by \cref{def:bko_exc}, and we find
\[
  \maxax(F_1) = 2, \qquad \maxax(F_2) = 3.
\]
As a result, $f$ is a counterexample to \cref{conj:bko}.
\end{example}

\begin{rem}
The facet $F_2$ in the previous example is a $B_2$-facet,
as defined by Esterov, Lemahieu and Takeuchi \cite{ELT_monodromy},
see also \cite{LPS_motivic,Selyanin_B}.
\end{rem}

\begin{block}
In order to improve \cref{conj:bko}, it is tempting to weaken
the condition of being exceptional as follows:
\emph{$F$ is weakly exceptional} if it has a triangulation consisting
of exceptional triangles.
In \cref{example:counter}, the facet $F_2$ is then weakly
exceptional, a triangulation is described in \cref{fig:decomposition}.
Note that the indices $i,j$ in the definition of exceptional cannot be
chosen the same for each simplex in this decomposition.
\begin{figure}[ht]
\begin{center}
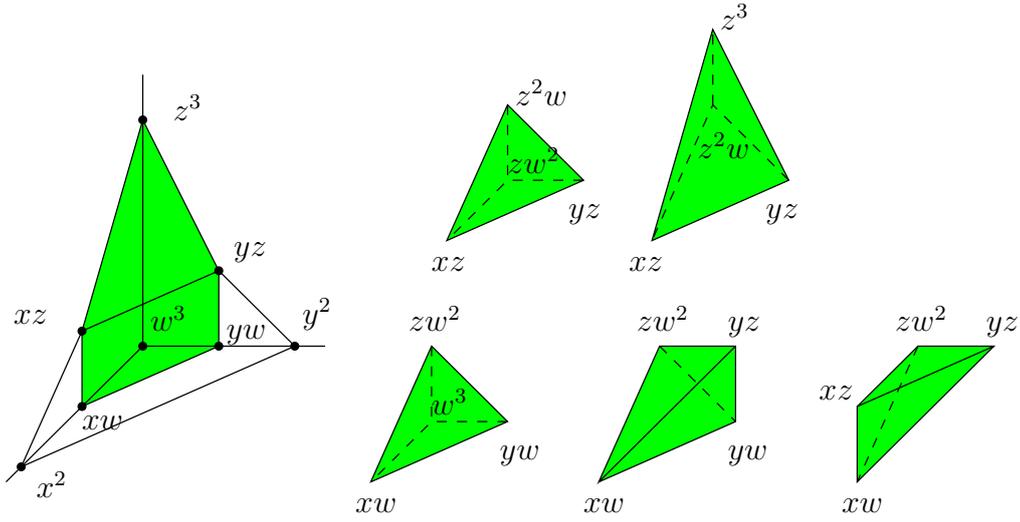
\caption{A decomposition of the facet $F_2$. The picture shows a
projection of the Newton diagram of $f$ to the $xyz$-coordinate space
along the $w$-axis.}
\label{fig:decomposition}
\end{center}
\end{figure}
However, if $f$ is a function with generic coefficients in front of
the monomials
\begin{equation} \label{eq:3ex}
  x^3, y^3, xz^2, xw^2, yz^2, yw^2, w^4, z^4,
\end{equation}
then we can find a similar decomposition which includes the simplex
with vertices
\[
  xzw, yzw, z^3w, z^3w
\]
which is not exceptional. Thus, the following question remains:
\end{block}

\begin{question} \label{question:bko}
Given only a facet $F$ of a Newton diagram $\Gamma = \Gamma(f)$,
how can one identify whether $F$ should be considered
\emph{exceptional} (without knowing $\Gamma$),
in such a way that the statement in
\cref{conj:bko} is true?
\end{question}

\begin{block}
We sketch a possible answer to \cref{question:bko},
relative to a triangulation $\T$ of the Newton diagram, following methods from
\cite[6.1]{selyanin_mono}.
Assume that we have chosen a triangulation $\T$ of $\Gamma(f)$.
If $S \subset \R^{n+1}$ is a simplex
whose affine hull does not contain the origin, denote by $S_-$
the convex hull of $S \cup \{0\}$. If $v_0,v_1,\ldots,v_s$ are the
vertices of $S$, set (see \cite[Notation 2.21]{selyanin_mono})
\[
  \capno(S)
  =
  \left|
    \set{\sum_{i=0}^s \lambda_i v_i}{0 < \lambda_i < 1}
    \cap
    \Z^{n+1}
  \right|.
\]
We also set
\[
  \capno(\emptyset) = 1.
\]
If $S \in \T$ is a coordinate simplex contained in a coordinate facet
$F$ of the same dimension, then we set $m_S = m_F$ and $n_S = n_F$.
This way, we have
\[
  m_S s! \Vol_s(S)
  = (s+1)! \Vol_{s+1}(S_-)
  = \sum_{\substack{ T \in \T \cup \{\emptyset\} \\ T \subset S}} \capno(T).
\]
Denote by $\T_c \subset \T$ the set of coordinate simplices. Then, using
$s = \dim(S)$, \cref{eq:AJ_vol} reads
\begin{equation} \label{eq:AJ_Cap}
\begin{split}
  \tTPS(f,0)
  &= \sum_{S \in \T_c}
  \frac{(-1)^{n-s}}{m_S}
    \sum_{\substack{ T \in \T \cup \{\emptyset\} \\ T \subset S}} \capno(T)
  \newpol{m_S}{n_S} \\
  &=
  \sum_{T \in \T \cup \{\emptyset\}}
  \capno(T)
  \sum_{\substack{ S \in \T_c \\ S \supset T}}
  \frac{(-1)^{n-s}}{m_S}
    \newpol{m_S}{n_S}
\end{split}
\end{equation}
\end{block}

\begin{definition} \label{def:CNC}
With $\T$ as above, and $T \in \T$, or $T = \emptyset$,
we define the \emph{relative combinatorial Newton polyhedron}
of $\T$ as
\[
  \CNN(\T/T)
  =
  \sum_{\substack{ S \in \T_c \\ S \supset T}}
  \frac{(-1)^{n-s}}{m_S}
    \newpol{m_S}{n_S}
\]
\end{definition}

\begin{block}
The above definition lifts the definition of the combinatorial Newton number
(see e.g. \cite[Definition 2.18]{selyanin_mono})
from an integral invariant of a
triangulation, to a Newton polygon, up to the inclusion of
the empty set in a triangulation.
Assuming first that $T = \emptyset$, then $\CNN(\T/T)$
is related to the combinatorial Newton number $\CN(\T)$ by the equality
\[
  (-1)^{n+1} + \ell(\CNN(\T))
  =
  \sum_{S \in \T_c \cup \{\emptyset\}} (-1)^{n-s}
  =
  \CN(\T).
\]
where we set $s = -1$ if $S = \emptyset$.
If $T \neq \emptyset$, then the length of $\CNN(\T/T)$ equals
the combinatorial Newton number $\CN(\T/T)$ of the link of $T$ in $\T$.
Assuming that $\Gamma(f)$ is convenient, the Newton number
$\CN(\T)$
is nonnegative by \cite{selyanin_mono}.
In fact, it equals the value of a local $h$-polynomial at $1$,
and this polynomial has nonnegative coefficients \cite{stanley_subdivisions}.
Note, however, that $\CNN(\T/T)$ does not always have nonnegative coefficients.

Since the alternating Jacobian polygon can be calculated in terms of
these combinatorial Newton polygons by \cref{eq:AJ_Cap}, we would
like to use this formula to obtain the \L ojasiewicz exponent of $f$.
Indeed, the degree on the left hand side of \cref{eq:AJ_Cap} is bounded
above by the maximal degree of the terms on the right hand side, i.e.
\begin{equation} \label{eq:AJ_Cap_degineq}
  \deg\tTPS(f,0)
  \leq
  \max \set{\deg \CNN(\T/T)}{
\begin{matrix}
    T \in \T\cup \{\emptyset\},\\
    \capno(T) \neq 0,\\
    \CNN(\T/T)\neq 0
\end{matrix}
  }.
\end{equation}
If equality holds here, then we do not have to include the case
$T = \emptyset$.
This is because, for any $S \in \T_c$, we have $\chi(S) = 1$, and so we find
\[
  \CNN(\T/\emptyset)
  = \sum_{T \in \T} (-1)^{\dim T} \CNN(\T/T).
\]
This is to say that, if equality holds in \cref{eq:AJ_Cap_degineq}, then
the same equality holds with the condition $T \in \T \cup \{\emptyset\}$
replaced by $T \in \T$.
If this is true, then we find that if $F \subset \Gamma(f)$
is a coordinate facet satisfying $\maxax(F) - 1 > \Loj(f,0)$, then
\[
  \maxax(F) > \deg(\CNN(\T/T))
\]
for all $T \in \T$, $T \subset F$.
Furthermore, there would necessarily exist some coordinate facet $F$,
and a $T \in \T$, $T \subset F$ with $\maxax(F) = \deg(\CNN(\T/T))$.
\end{block}

\begin{definition} \label{def:FTne}
Let $f \in \C\{x_0, x_1, \ldots, x_n\}$ be a Newton nondegenerate
function germ having an isolated singularity at $0$, and
$\T$ a triangulation of $\Gamma(f)$.
\begin{enumerate}
\item
Denote by $\Tne$ the set of simplices in $\T$ satisfying
\[
  \capno(T) \CNN(\T/T) = 0.
\]
\item
Denote by $\CFne^\T$ the set of
those coordinate facets $F \in \CF$ which contain a simplex
$T \in \Tne$ satisfying
\[
  \deg(\CNN(\T/T)) \geq \maxax(F).
\]
\end{enumerate}
\end{definition}

\begin{conj}
Let $f \in \C\{x_0, x_1, \ldots, x_n\}$ be a Newton nondegenerate
function germ having an isolated singularity at $0$.
Assume that
$f$ does not have a Morse point at $0$ or that $n$ is even.
If $\T$ is a triangulation of $\Gamma(f)$,
then
\[
  \Loj(f,0)
  = \max_{T \in \T_{\mathrm{ne}}} \deg(\CNN(\T/T)) - 1
  = \max_{F \in \CF_{\mathrm{ne}}^\T} \maxax(F) - 1.
\]
\end{conj}

\begin{example} \label{ex:twosimp}
Consider the Morse point in three variables
\[
  f(x,y,z) = xy + xz + 2yz + z^2
\]
and the two triangulations of the Newton diagram $\Gamma(f)$
seen in \cref{fig:twosimp}, with vertices
\[
  A:yz,\quad
  B:z^2,\quad
  C:xz,\quad
  D:xy.
\]
On the left hand side, the set $\Tne$ contains only the blue
triangle $ACD$. As a result, $F = ABCD$ is the only element
of $\CFne$.
On the right hand side, however, $\Tne$ consists of
only the vertex $B$, and $\CFne$ contains all coordinate facets.

\begin{figure}[ht]
\begin{center}
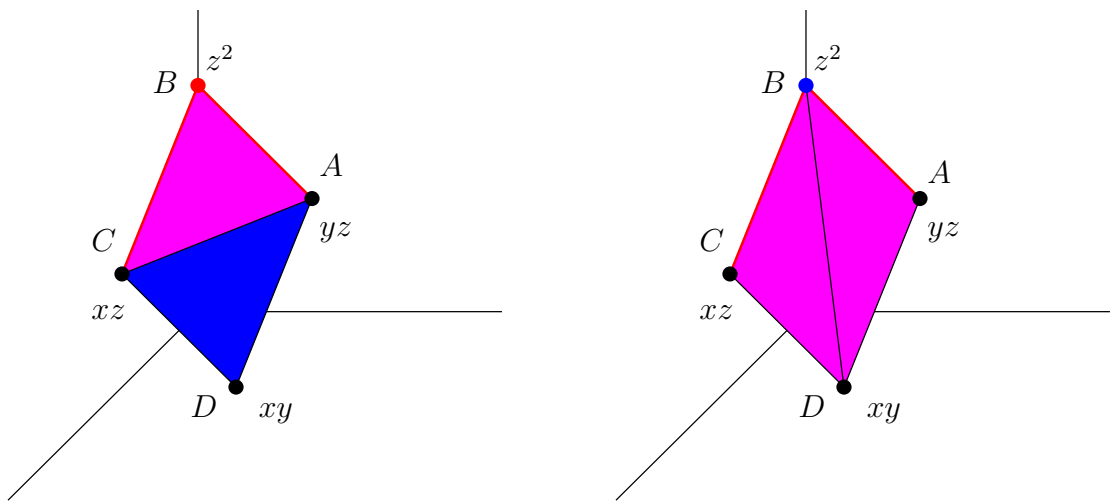
\caption{Two triangulations of the Newton diagram $\Gamma(f)$.}
\label{fig:twosimp}
\end{center}
\end{figure}

\end{example}

\bibliographystyle{alpha}
\bibliography{bibliography}

\end{document}